\documentclass[a4paper,11pt]{amsart}
\usepackage[T1]{fontenc}
\usepackage{soul}
\usepackage{amssymb}
\usepackage{enumitem}
\usepackage{lmodern}
\usepackage{mathtools}
\usepackage[]{microtype}
\usepackage{nicefrac}
\usepackage[initials, alphabetic]{amsrefs}
\usepackage[hidelinks]{hyperref}
\usepackage{cleveref}

\numberwithin{equation}{section}
\newtheorem{theorem}{Theorem}[section]
\newtheorem{corollary}[theorem]{Corollary}
\newtheorem{lemma}[theorem]{Lemma}

\newtheorem{proposition}[theorem]{Proposition}

\newtheorem*{theorem*}{Theorem}

\newtheorem{theoremintro}{Theorem}

\theoremstyle{definition}
\newtheorem{definition}[theorem]{Definition}
\newtheorem{remark}[theorem]{Remark}

\newtheorem*{problem*}{Problem}

\newcommand{\bbC}{\mathbb{C}}

\newcommand{\cD}{\mathcal{D}}

\newcommand{\bbN}{\mathbb{N}}

\newcommand{\cP}{\mathcal{P}}

\newcommand{\bbR}{\mathbb{R}}

\newcommand{\cU}{\mathcal{U}}

\newcommand{\cV}{\mathcal{V}}

\newcommand{\cW}{\mathcal{W}}

\newcommand{\bbZ}{\mathbb{Z}}

\newcommand{\cZ}{\mathcal{Z}}

\newcommand{\e}{\varepsilon}

\renewcommand{\phi}{\varphi}

\newcommand{\rank}{\mathrm{rank}}
\newcommand{\sa}{\mathrm{sa}}

\title[Stable rank one, local homogeneity and uniform property $\Gamma$]{Stable rank one, tracial local homogeneity and uniform property $\Gamma$}

\author{Andrea Vaccaro}
\address{Andrea Vaccaro, Institut Camille Jordan, Université Claude Bernard Lyon 1, 43 Boulevard du 11 novembre 1918, 69622, Villeurbanne, France}
\email{vaccaro@math.univ-lyon1.fr}
\urladdr{https://sites.google.com/view/avaccaro}

\begin{document}
\maketitle

\vspace{-1.5em}

\begin{abstract}
We prove that separable, simple, unital, non-elementary, stably finite $C^\ast$-algebras that have stable rank one, and that have locally finite nuclear dimension in a tracial sense, have uniform property $\Gamma$. In particular, Villadsen algebras of the first type and crossed products of free minimal actions of FC (in particular, abelian) groups on compact metric spaces have uniform property $\Gamma$. This implies that all these $C^\ast$-algebras satisfy the Toms--Winter conjecture, a fact already known for $C^\ast$-algebras with stable rank one and locally finite nuclear dimension, and here recovered via a different approach.
\end{abstract}

\vspace{1em}

In their celebrated last work on rings of operators \cite{MvN:IV}, Murray and von Neumann used the existence of non-trivial asymptotically central sequences as an invariant to distinguish the hyperfinite II$_1$ factor from the free group factor, thereby providing the first examples of non-isomorphic II$_1$ factors. The notion indicating the existence of such sequences was simply labelled \emph{property $\mathit{\Gamma}$}, being one among several properties considered in the paper---the others expressing equivalent formulations of what now would be called \emph{hyperfiniteness}---and indexed by letters.

Property $\Gamma$ would later come to play a key role in the subject, thanks to the emergence of central sequence algebras as a fundamental tool in the theory, beginning with the work of McDuff \cite{mcduff:central}, and it has since remained an important concept in von Neumann algebras. This notion eventually found its way into the setting of stably finite $C^\ast$-algebras, with an adaptation, \emph{uniform property $\Gamma$} (\Cref{def:gamma}), introduced in \cite{CETWW} as an auxiliary tool used to confirm one of the implications of the Toms--Winter conjecture.

Although its formulation is, \emph{mutatis mutandis}, formally close to that of its counterpart for II$_1$ factors, uniform property $\Gamma$ underlies a much more complex and less apparent form of regularity.
Broadly speaking, uniform property $\Gamma$ expresses a powerful form of approximately central tracial divisibility, and it has been observed to have strong effects on the structure of the $C^\ast$-algebras that satisfy it (see \cite{CCEGSTW}). Shortly after its introduction, its connection to the Toms--Winter conjecture---which asserts that for separable, simple, non-elementary, nuclear $C^\ast$-algebras the three key regularity conditions $\cZ$-stability, finite nuclear dimension and strict comparison are equivalent (see \cite[Conjecture 5]{STW:99problems})---was further reinforced in \cite{CETW:gamma}, where it was  discovered that the conjecture holds true under the assumption of uniform property $\Gamma$, naturally prompting the following question.

\begin{problem*}[{\cite[Question C]{CETW:gamma}, \cite[Problem XIX]{STW:99problems}}]
Do all separable, simple, unital, non-elementary (i.e. non-isomorphic to the compact operators) nuclear, stably finite $C^\ast$-algebras have uniform property $\Gamma$?
\end{problem*}

This problem remains open and, to the best of our knowledge, there are essentially three, or rather four, routes to establishing uniform property $\Gamma$.

The first---and farthest from optimal---is assuming $\cZ$-stability (see \cite[Proposition 2.3]{CETWW}). The second, possible if the given $C^\ast$-algebra is nuclear, is requiring that the trace space is \emph{tractable} in a topological sense---more precisely, that the space of extremal traces is compact and has finite covering dimension. This was proved in \cite{KR:central, TWW:fin_dim, Sato:traces} and more recently generalized under weaker assumptions even beyond the nuclear setting in \cite{ES:gamma}. The third method---perhaps the least known among non-experts---is highlighted in \cite[Theorem 5.5]{CETW:gamma} and is due to Winter, who proved in \cite{Winter:pure} that tracial almost divisibility (see \Cref{def:tad}) implies uniform property $\Gamma$ for simple $C^\ast$-algebras that have locally finite nuclear dimension, that is $C^\ast$-algebras that can be locally approximated by $C^\ast$-algebras with finite nuclear dimension (\Cref{def:loc_fin_nucdim}).

The fourth and last route only applies to $C^\ast$-algebras that are crossed products of topological dynamical systems, and it is due to Kerr and Szab\'o. They proved in \cite{KS:sbp} that free actions of countably infinite amenable groups on compact metric spaces generate crossed products with uniform property $\Gamma$, if they satisfy the \emph{small boundary property} (see \cite[Definition 5.1]{KS:sbp}), a dynamical formulation of zero-dimensionality which was later shown to be actually equivalent to a relative version of uniform property $\Gamma$ for the Cartan pairs generated by the crossed products \cite{KLTV}.

The following question points to a remarkable blind spot in the state of the art: are there examples of simple nuclear $C^\ast$-algebras with uniform property $\Gamma$ that are non-$\cZ$-stable \emph{and} have intractable trace space? A natural test case is the class of non-$\cZ$-stable Villadsen algebras of the first type from \cite{Villadsen:1}---their trace space is the Poulsen simplex, by \cite[Theorem 4.5]{ELN:Villadsen}---for which this problem has been open for some years and more recently advertised in \cite[Problem XX]{STW:99problems}.
Other natural candidates to consider are the non-$\cZ$-stable crossed products of free minimal $\bbZ$-actions obtained in \cite{GiolKerr}, which cannot have tractable trace spaces by \cite[Theorem 1.4]{EN:sbp}.

In this paper we finally settle the problem of uniform property $\Gamma$ for these and other $C^\ast$-algebras, solving affirmatively \cite[Problem XX]{STW:99problems}.
\begin{theoremintro}[\Cref{cor:vill}, \Cref{cor:cp}] \label{thm_main:Vill_cp}
Suppose that $A$ is either
\begin{enumerate}
\item A simple, unital, non-elementary AH-algebra with stable rank one, for instance a non-elementary Villadsen algebra of the first type.
\item A crossed product of a free minimal action by a countably infinite FC (e.g. abelian) group on a compact metric space.
\end{enumerate}
Then $A$ has uniform property $\Gamma$.
\end{theoremintro}


\Cref{thm_main:Vill_cp} is a corollary of the following more general result, which permits to establish uniform property $\Gamma$ for a vast class of nuclear $C^\ast$-algebras, under the assumption of stable rank one.

\begin{theoremintro}[\Cref{cor:sr1_tad_gamma}] \label{thm_main:sr1_gamma}
Let $A$ be a separable, simple, unital, non-elementary, stably finite $C^\ast$-algebra with stable rank one and tracially locally finite nuclear dimension (\Cref{def:loc_fin_nucdim}). Then $A$ has uniform property $\Gamma$.
\end{theoremintro}

We remark that, by \cite[Theorem A]{CETW:gamma}, the $C^\ast$-algebras in \Cref{thm_main:sr1_gamma} verify the Toms--Winter conjecture. This was already proved in \cite{Thiel:ranks} for $C^\ast$-algebras with stable rank one and locally finite nuclear dimension, and it is recovered here with a different proof using uniform property $\Gamma$.

The argument for \Cref{thm_main:sr1_gamma} is composed of three main steps. The first is contained in the recent preprint \cite{Fu:sr1_rr0}, where it is proved that separable, simple, unital $C^\ast$-algebras with stable rank one have  \emph{tracial approximate oscillation zero}, which is equivalent to say that their tracial ultrapowers have real rank zero. This work refines the earlier \cite{FL:sr1}, where the same statement is deduced under the additional assumption of strict comparison. 

 The second ingredient---a tracial variant of a result due to Zhang from \cite{Zhang:rr0}---is \Cref{prop:rr0_tad}, where it is shown that if $A$ is a simple, unital, non-elementary, stably finite $C^\ast$-algebra whose tracial ultrapower has real rank zero, then $A$ is tracially almost divisible.
 
Finally, the last ingredient, already mentioned in this introduction, is Winter's result from \cite{Winter:pure} stating that tracial almost divisibility, combined with locally finite nuclear dimension, implies uniform property $\Gamma$.
Winter's argument also applies to $C^\ast$-algebras that have \emph{tracially} locally finite nuclear dimension, namely those $C^\ast$-algebras for which every finite subset can be approximated, in the uniform tracial norm, by some subalgebra with finite nuclear dimension, bringing us back to the statement of \Cref{thm_main:sr1_gamma}.

We remark that it is not known whether nuclear $C^\ast$-algebras that do not have locally finite nuclear dimension exist, so the reduction to tracial locally finite nuclear dimension could potentially be---at least in the nuclear setting---inconsequential\footnote{Note however that $C^\ast$-algebras  that have {tracially} locally finite nuclear dimension need not be nuclear: take, for instance, the hyperfinite II$_1$ factor.}. Nevertheless, tracially locally finite nuclear dimension comes in handy when deducing the second part of \Cref{thm_main:Vill_cp} from \Cref{thm_main:sr1_gamma}.

Indeed, while the first part of \Cref{thm_main:Vill_cp} immediately follows from \Cref{thm_main:sr1_gamma}---AH-algebra clearly have locally finite nuclear dimension, being locally homogeneous, and Villadsen algebras of the first type are well-known to have stable rank one by \cite[Proposition 10]{Villadsen:1}---the author is not aware of a straightforward proof showing that crossed products of dynamical systems have locally finite nuclear dimension. Nevertheless, a technical result due to Niu from \cite{niu:rc_mdim_urp} (which we will review in detail in \Cref{prop:LTH_crossed}) shows that actions with the \emph{uniform Rokhlin property} (abbreviated \emph{URP}, see \Cref{def:urp}) generate crossed products that enjoy a tracial form of approximate homogeneity, from which tracially locally finite nuclear dimension immediately follows.

To finish the proof of the second part of \Cref{thm_main:Vill_cp}, one can finally rely on the  breakthrough paper \cite{LN:sr1} (see also \cite{BGK:sr1}), which establishes stable rank one for crossed products of free minimal topological dynamical system that have the URP and \emph{comparison} in the sense of \cite[Definition 3.2]{kerr:af}. The fact that the combination of these two conditions is automatic for actions as in \Cref{thm_main:Vill_cp} is but one of the many striking results obtained in \cite{naryshkin:urpc}.

The present version of this manuscript benefits from the recent \cite{Fu:sr1_rr0}, of which the author became aware only at a later stage of the project. In an earlier version, real rank zero of the tracial ultrapower was obtained via a generalization of an argument of topological flavor sketched in \cite[Example 3.5]{EN:propS} which does not use stable rank one, but relies instead on a tracial version of local homogeneity with locally controlled dimension growth, introduced in \Cref{def:LTH}. This local homogeneity with locally controlled dimension growth is enjoyed by certain AH-algebras (including Villadsen algebras of the first type, see \Cref{remark:Vill}) and by crossed products of free minimal topological dynamical systems with the URP, and provides an alternative route to uniform property $\Gamma$ which is presented in the second half of the paper (see \Cref{thm:LTH_rr0}, \Cref{thm:LTH_gamma}). We remark that this approach allows to deduce uniform property $\Gamma$ also for $C^\ast$-algebras which a priori do not have stable rank one (see \Cref{cor:AH_fdg}).

As a final observation, we emphasize that \Cref{cor:cp}, where we prove that free minimal topological dynamical systems with the URP generate crossed products with uniform property $\Gamma$, recovers with a new proof (but under the additional assumption of minimality) the theorem from \cite{KS:sbp} stating the same implication for free actions by amenable groups with the small boundary property. 

To better explain this, it is instructive to draw a parallel between Kerr and Szab\'o's proof and Winter's argument in \cite{Winter:pure} that tracial almost divisibility and locally finite nuclear dimension imply uniform property $\Gamma$. In both cases, one starts with some non-central version of tracial divisibility---tracial almost divisibility in Winter’s work, and the small boundary property in Kerr and Szab\'o’s---and upgrades it to an approximately central form of divisibility---\emph{almost finiteness in measure} in \cite{KS:sbp}---via a strong variant of amenability. In Winter's case the strong variant of amenability is  locally finite nuclear dimension, as opposed to nuclearity, while in \cite{KS:sbp} this comes from the amenability of the acting group, rather than merely that of the action, which alone would not imply a F{\o}lner-type condition like almost finiteness in measure.

In summary, \Cref{cor:cp} recovers \cite[Theorem 9.4]{KS:sbp} through Winter's method, crucially leveraging the small boundary property---or just the weaker URP---to obtain tracial almost divisibility and tracially locally finite nuclear dimension, while completely bypassing almost finiteness in measure (a similar approach, using comparison, can be found in \cite{niu:transformation, LN:sr1}).

\subsection*{Summary of the paper} Preliminaries are covered in \S\ref{S:preliminaries}. In \S\ref{S:rr0} we show that having a tracial ultrapower with real rank zero implies tracial almost divisibility for simple $C^\ast$-algebras, and thus uniform property $\Gamma$ if in addition tracially locally finite nuclear dimension is assumed. We then use this to prove \Cref{thm_main:sr1_gamma}.
In \S\ref{S:LTH} we introduce tracially LTH-algebra with locally flat dimension growth and we prove that they have, when simple, tracial ultrapowers with real rank zero, from which we deduce uniform property $\Gamma$. In \S\ref{S:LTH_examples} we put all pieces together and prove \Cref{thm_main:Vill_cp}.

\subsection*{Acknowledgements} I am particularly grateful to Jamie Bell for bringing \cite{Fu:sr1_rr0} to my attention, which led to the current form of the main results. I thank Zhuang Niu for some insightful remarks and for pointing out to me that some additional assumptions in a previous version of \Cref{thm:LTH_rr0} and \Cref{thm:LTH_gamma} could be removed. I wish to thank Chris Schafhauser, Aaron Tikuisis and Stuart White for the valuable feedback at the early stages of the project and on the first draft of this manuscript. The idea of interpreting crossed products as \emph{tracially} approximately homogeneous algebras, which  eventually led to \Cref{def:LTH}, first emerged after some conversations I had with Grigoris Kopsacheilis. I would like to thank him for those and other insightful discussions we shared. A large portion of this paper has been written during a visit to Universit\"at M\"unster in April 2026. I would like to express my gratitude to the institute and all the members of the operator algebras research group there for the warm welcome, the stimulating conversations and for the encouraging feedback on this work.

\section{Preliminaries} \label{S:preliminaries}
Let $A$ be a $C^\ast$-algebra. Denote by $A_1$, $A_\sa$ and $A_+$ the sets of contractions, self-adjoint elements, and positive elements of $A$, respectively. We also use the notation $A_{1, +}$ and $A_{1, \sa}$ to denote positive and self-adjoint contractions. We let $1_A$ denote the multiplicative unit of $A$, if $A$ has one. In case $A$ is equal to some matrix algebra $M_n$, we abbreviate $1_{M_n}$ as $1_n$. Given $a, b \in A$, we abbreviate the commutator $ab - ba$ with $[a,b]$.

Let $T(A)$ be the set of all tracial states---which we simply call \emph{traces}---endowed with the weak$^\ast$-topology induced by the dual of $A$. If $A$ is unital, then $T(A)$ is compact with this topology.

Given $\tau \in T(A)$, the definition of the seminorm induced by $\tau$ on $A$ is
\begin{equation}
\| a \|_{2, \tau} \coloneqq \tau(a^\ast a)^{\/2}, \quad a \in A,
\end{equation}
and
\begin{equation}
\| a \|_{2, T(A)} \coloneqq \sup_{\tau \in T(A)} \| a \|_{2, \tau}.
\end{equation}
We recall that
\begin{equation}
\| ab \|_{2, T(A)} \le \| a \| \|b \|_{2, T(A)}, \quad a,b \in A.
\end{equation}
We will often use this inequality without mention.

Given $a \in A$, recall that the smallest ideal of $A$ containing $a$ is $\overline{AaA}$, and that the smallest hereditary subalgebra of $A$ containing $a$ is $\overline{aAa}$.

Recall finally that a unital $C^\ast$-algebra $A$ has \emph{stable rank one} if the set of invertible elements in $A$ is dense in $A$, and it has \emph{real rank zero} if the set of invertible elements of $A_\sa$ is dense in $A_\sa$. Among the many well-known equivalent characterizations of real rank zero, we will regulary use without mention the one stating that every hereditary subalgebra of $A$ has an approximate unit of projections \cite[Theorem 2.6]{BP:rr0}.

\subsection{Tracial ultrapower and uniform property $\Gamma$}
Fix, for the rest of this paper, a non-principal ultrafilter $\cU$ over $\bbN$.

Let $A$ be a $C^\ast$-algebra such that $T(A) \ne \emptyset$. Its \emph{tracial ultrapower} is the quotient
\begin{equation}
A^\cU \coloneqq \ell^\infty(A) /J_{T(A)},
\end{equation}
where $J_{T(A)}$ is the \emph{trace-kernel ideal}, namely
\begin{equation}
J_{T(A)} \coloneqq \left\{ (a_n)_{n=1}^\infty \in A^\cU : \lim_{n \to \cU} \| a_n \|_{2, T(A)} = 0 \right\}.
\end{equation}
Similarly to how we just did in the definition of $J_{T(A)}$, will regularly identify sequences in $\ell^\infty(A)$ with their equivalence classes in $A^\cU$. We recall that $A^\cU$ is unital if and only if $T(A)$ is compact, which is always the case when $A$ itself is unital (see \cite[Proposition 1.11]{CETWW}).

The map sending each $a \in A$ to the constant sequence in $A^\cU$ with value $a$ is an embedding whenever $\| \cdot \|_{2, T(A)}$ is a norm---e.g. if $A$ is simple. We shall regularly identify $A$ with this canonical copy of it in $A^\cU$.

Every sequence $ \bar \tau \coloneqq (\tau_n)_{n=1}^\infty$ of traces on $A$ induces a trace on $A^\cU$ by setting
\begin{equation}
\bar \tau((a_n)_{n=1}^\infty) \coloneqq \lim_{n \to \cU} \tau_n(a_n), \quad (a_n)_{n=1}^\infty \in A^\cU.
\end{equation}
We call traces of this form \emph{limit traces}. Denote by $T_\cU(A) \subseteq T(A^\cU)$ the weak$^\ast$-closed convex hull of all limit traces in $T(A^\cU)$.

Finally, given a subset $S \subseteq A^\cU$, we denote by $A^\cU \cap S'$ the commutant of $S$ in $A^\cU$.

\begin{definition}[{\cite[Definition 2.1]{CETWW}}] \label{def:gamma}
Let $A$ be a separable $C^\ast$-algebra such that $T(A)$ is non-empty and compact. We say that $A$ has \emph{uniform property $\mathit{\Gamma}$} if, for every $n \ge 1$, there are pairwise orthogonal projections $p_0, \dots, p_{n-1} \in A^\cU \cap A'$ such that $\sum_{i < n} p_i = 1_{A^\cU}$ and
\begin{equation}
\tau(ap_i) = \frac{1}{n} \tau(a), \quad i < n, \  a \in A, \ \tau \in T_\cU(A).
\end{equation}
\end{definition}

\begin{definition}[{\cite[Definition 4.2]{CETW:gamma}}] \label{def:mcduff}
Let $A$ be a separable $C^\ast$-algebra such that $T(A)$ is non-empty and compact. We say that $A$ is \emph{uniformly McDuff} if, for every $n \ge 1$, there is a unital $*$-homomorphism $\Phi \colon M_n \to A^\cU \cap A'$.
\end{definition}

Uniform McDuffness is formally stronger than uniform property $\Gamma$, but the two conditions are equivalent in the nuclear setting (see \cite[Proposition 2.3]{CETWW} and \cite[Theorem 4.6]{CETW:gamma} for details).

\subsection{Cuntz subequivalence and dimension functions} We assume some familiarity with the basics of Cuntz subequivalence (see e.g. \cite{thiel:cuntz}).

Let $A$ be a $C^\ast$-algebra. Given $a, b \in A_+$, we say that $a$ is \emph{Cuntz-subequivalent} to $b$, $a \preceq b$ in symbols, if and only if there is a sequence $(x_n)_{n=1}^\infty$ of elements in $A$ such that $\lim_{n \to \infty} \| x_n b x_n^\ast - a \| = 0$. We say that $a$ and $b$ are \emph{Cuntz-equivalent}, $a\sim b$ in symbols, if $a \preceq b$ and $b \preceq a$. Recall that if $p,q \in A$ are projections, then $p$ is Cuntz-subequivalent (equivalent) to $q$ if and only if $p$ is subequivalent (equivalent) to $q$, in the sense that there is a partial isometry $v \in A$ such that $p = v^\ast v$ and $vv^\ast \le q$ ($vv^\ast = q$).

If $ \tau \in T(A)$, define the \emph{dimension function} associated to $\tau$ as
\begin{equation}
d_\tau(a) \coloneqq \lim_{n \to \infty} \tau(a^{1/n}), \quad a \in A_+.
\end{equation}
 Recall finally that $a \preceq b$ implies $d_\tau(a) \le d_\tau(b)$.

\subsection{Tracial almost divisibility and finite nuclear dimension}
Given two $C^\ast$-algebras $A$ and $B$, a {completely positive} (c.p.) map $\phi \colon A \to B$ is \emph{order zero} if it preserves orthogonality of positive elements, that is $\phi(a)\phi(b) = 0$ whenever $a,b \in A_+$ are such that $ab = 0$.

The notion of tracial almost divisibility appeared first in \cite[Definition 3.5]{Winter:pure} as a tracial counterpart to the more well-known almost divisibility. We report the definition below, modifying it slightly to also express tracial almost divisibility for single positive contractions, which will be useful in the next section.

\begin{definition}[{cf. \cite[Definition 3.5]{Winter:pure}}] \label{def:tad}
Let $A$ be a unital $C^\ast$-algebra such that $T(A) \ne \emptyset$. We say that $a \in A_{1,+}$ is \emph{tracially almost divisibile in $A$} if, for every $n \ge 1$ and $\e > 0$, there is a completely positive contractive (c.p.c.) order zero map
\begin{equation}
\phi \colon M_n \to \overline{a A a},
\end{equation}
such that $\tau(\phi(1_n)) \ge \tau(a) - \e$ for all $\tau \in T(A)$.

A unital tracial $C^\ast$-algebra $A$ is \emph{tracially almost divisible} if every $a \in M_m(A)_{1,+}$ is tracially almost divisible in $M_m(A)$, for all $m \in \bbN$.\footnote{The original definition of tracial almost divisibility in \cite[Definition 3.5]{Winter:pure} is stated in terms of quasitraces. We avoid mentioning quasitraces here since we will mainlt interested in unital nuclear $C^\ast$-algebras, for which it is known that all 2-quasitraces are traces, by the celebrated \cite{haagerup:quasi}.}
\end{definition}

\begin{remark} \label{remark:matrix}
Recall that the existence of a c.p.c. order zero map $\phi \colon M_n \to A$ is equivalent to the existence of pairwise orthogonal elements $a_0, \dots, a_{n-1} \in A_{1,+}$ for which there are $x_0, \dots, x_{n-1} \in A_1$ such that $x_i^\ast x_i = a_i$ and $x_i x_i^\ast = a_1$, for all $i < n$ (see e.g. \cite[Propositions 2.3, 2.4]{RW:revisited}). It is moreover well-known that the existence of a $*$-homomorphism $\Phi \colon M_n \to A$ is equivalent to the existence of pairwise orthogonal equivalent projections $p_0, \dots, p_{n-1} \in A$.
\end{remark}

The \emph{nuclear dimension} of a $C^\ast$-algebra $A$ (see  \cite{WZ:nuc_dim}) is the smallest integer $n \in \bbN$, denoted $\dim_{\mathrm{nuc}}(A)$, such that there is a net $(F_i, \psi_i, \phi_i)_{i\in I}$ where each $F_i$ is a finite-dimensional $C^\ast$-algebra, $\psi_i \colon A \to F_i$ is a c.p.c. map, $\phi_i \colon F_i \to A$ is a c.p map which is sum of at most $n+1$ c.p.c.   order zero maps, and the net $(\| \phi_i(\psi_i(a)) - a \|)_{i \in I}$ converges to zero. If no such integer exists we say that $A$ has infinite nuclear dimension. 

\begin{definition}[{cf. \cite[Definition 4.1]{Winter:pure}}] \label{def:loc_fin_nucdim}
A $C^\ast$-algebra has \emph{locally finite nuclear dimension} if for every finite subset $F \subset A$ and $\e > 0$ there exists a $C^\ast$-subalgebra $B \subseteq A$ such that $\dim_{\mathrm{nuc}}(B) < \infty$ and that for every $a \in F$ there is $b \in B$ such that $\| a - b \| < \e$.

A $C^\ast$-algebra with non-empty trace space has \emph{tracially locally finite nuclear dimension} if for every finite subset $F \subset A_1$ and $\e > 0$ there exists a $C^\ast$-subalgebra $B \subseteq A$ such that $\dim_{\mathrm{nuc}}(B) < \infty$ and that for every $a \in F$ there is $b \in B_1$ such that $\| a - b \|_{2, T(A)} < \e$.
\end{definition}

The following proposition, crucially relying on \cite[Lemma 5.11]{Winter:pure}, will be our main gateway to uniform property $\Gamma$.

\begin{proposition}[{cf. \cite[Lemma 5.11]{Winter:pure}}] \label{prop:gamma_tad}
Let $A$ be a separable, simple, unital, non-elementary $C^\ast$-algebra such that $T(A) \ne \emptyset$. Consider:
\begin{enumerate}
\item \label{item1:gamma_tad} $A$ is uniformly McDuff.
\item \label{item2:gamma_tad} $A$ is tracially almost divisible.
\end{enumerate}
Then \ref{item1:gamma_tad} $\Rightarrow$ \ref{item2:gamma_tad}. If $A$ has tracially locally finite nuclear dimension, then \ref{item2:gamma_tad} $\Rightarrow$ \ref{item1:gamma_tad}.
\end{proposition}
\begin{proof}
\ref{item1:gamma_tad} $\Rightarrow$ \ref{item2:gamma_tad}. Since $M_m(A)^\cU \cong M_m(A^\cU)$ for every $m \in \bbN$, it follows that $A$ is uniformly McDuff if and only if $M_m(A)$ is, so it suffices to show that every element in $A_{1,+}$ is tracially almost divisibile.

Fix  then $a \in A_{1,+}$, $\e > 0$ and $n \ge 1$. Let $\Phi \colon M_n \to A^\cU \cap A'$ be a unital $*$-homomorphism. It then follows that the map
\begin{equation}
\begin{split}
\phi \colon M_n &\to \overline{aA^\cU a} \\
c &\mapsto a^{1/2} \Phi(c)a^{1/2}
\end{split}
\end{equation}
is a c.p.c. order zero map.

Let $\Pi \colon \overline{a \ell^\infty(A) a} \to \overline{aA^\cU a}$ denote the restriction of the quotient map $\ell^\infty(A) \to A^\cU$. As $\Pi$ is surjective and since $\overline{a \ell^\infty(A) a} = \ell^\infty(\overline{aAa})$, by \cite[Proposition 1.2.4]{Winter:cd_II} there is a sequence $(\phi_m)_{m=1}^\infty$ of c.p.c order zero maps from $M_n$ into $\overline{a A a}$ such that $\Pi \circ (\phi_m)_{m=1}^\infty = \phi$, where $(\phi_m)_{m=1}^\infty$ is interpreted as a map from $M_n$ into $\ell^\infty(A)$. 

Since, by definition, $\phi(1_n) = a$, there is $m \in \bbN$ such that $\tau(\phi_m(1_n)) >  \tau(a) - \e$ for all $\tau \in T(A)$.

\ref{item2:gamma_tad} $\Rightarrow$ \ref{item1:gamma_tad}. By \cite[Lemma 5.11]{Winter:pure}, for every $C^\ast$-subalgebra $B \subseteq A$ such that $\dim_{\mathrm{nuc}}(B) < \infty$, every finite subset $F \subset B$, $\e > 0$ and $n \ge 1$, there exists a c.p.c. order zero map $\phi \colon M_n \to A$ such that
\begin{equation}
\| [\phi(c), b] \| < \e \| c \|, \text{ and } \tau(\phi(1_n)) > 1 - \e, \quad c \in M_n, b \in F, \tau \in T(A).
\end{equation}

Repeating this argument for larger and larger subsets $F$ and smaller and smaller constants $\e$, by \cite[Lemma 7.6]{KR:central} one can obtain a unital $*$-homomorphism $\Phi \colon M_n \to A^\cU \cap B'$. Since $A$ has tracially locally finite nuclear dimension, the union of all $C^\ast$-subalgebras of $A$ with finite nuclear dimension of $A$ is dense, with respect to $\| \cdot \|_{2, T(A)}$, in $A$. By a standard application of Kirchberg's $\e$-test, there exists of a unital $*$-homomorphism $\Psi \colon M_n \to A^\cU \cap A'$ (see \cite[Proposition 5.23]{CCEGSTW} for the same argument in the context of tracially complete $C^\ast$-algebra), hence $A$ is uniformly McDuff.
\end{proof}

\section{Tracial ultrapowers with real rank zero and tracial almost divisibility} \label{S:rr0}
In this section we  show that simple, unital, non-elementary $C^\ast$-algebras whose tracial ultrapower has real rank zero are tracially almost divisible (see \Cref{def:tad}). This part is inspired by \cite[Theorem 1.1]{Zhang:rr0}, which proves a non-tracial analogue of this implication (see also \cite[Theorem 2.4]{KW:ample} for a similar statement for diagonal pairs). We then use this result to deduce \Cref{thm_main:sr1_gamma}, in \Cref{cor:sr1_tad_gamma}.

\begin{lemma} \label{lemma:rr0_tad}
Let $A$ be a unital $C^\ast$-algebra with real rank zero and such that $T(A) \ne \emptyset$. Let $a \in A_{1,+}$ and suppose that, for every $\delta > 0$ there is $b \in A_+$ such that $a \in \overline{AbA}$, and that
\begin{equation}
d_\tau(b) < \delta, \quad \tau\in T(A).
\end{equation}
Then, for every $n \ge 1$ and $\e > 0$, there exists a $*$-homomorphism $\Phi \colon M_n \to \overline{aAa}$,
such that
\begin{equation}
\tau(\Phi(1_n)) \ge \tau(a) - \e, \quad \tau \in T(A).
\end{equation}
In particular, $a$ is tracially almost divisible in $A$.
\end{lemma}
\begin{proof}
Fix $a \in A_{1,+}$, $\e > 0$ and $n \ge 1$. By real rank zero of $A$, the hereditary subalgebra $\overline{aAa}$ admits an approximate unit of projections, so there is a projection $p \in \overline{aAa}$ such that $\| pap - a \| < \e/2$, which in particular entails
\begin{equation} \label{eq:trace_p}
\tau(p) \ge \tau(pap) \ge  \tau(a) - \e/2, \quad \tau \in T(A).
\end{equation}

Find, using the hypothesis, some $b \in A_+$ such that $a \in \overline{AbA}$ and that
\begin{equation} \label{eq:small_b}
d_\tau(b) < \frac{\e}{2n}, \quad \tau \in T(A).
\end{equation}
As $p \in \overline{aAa} \subseteq \overline{AbA}$, there are finitely many $x_j, y_j \in A$ such that
\begin{equation}
\left\| \sum\nolimits_{j < m} x_j b y_j - p \right\| < 1.
\end{equation}
Since $\overline{b A b}$ has real rank zero, there is a projection $q \in \overline{b A b}$ such that $\| qb - b \|$ is small enough that
\begin{equation} \label{eq:p_q}
\left\| \sum\nolimits_{j < m} x_j q b y_j - p \right\| < 1.
\end{equation}
Notice that $q \in \overline{b Ab}$ implies $q \preceq b$ (see e.g. \cite[Theorem 2.43]{thiel:cuntz}), therefore
\begin{equation} \label{eq:small_q}
\tau(q) = d_\tau(q) \le d_\tau(b) \stackrel{\eqref{eq:small_b}}{<} \frac{\e}{2n}, \quad \tau \in T(A).
\end{equation}

The inequality in \eqref{eq:p_q} entails that $\sum_{j < m} x_j q b y_j$ is invertible in $p A p$, which in turn implies that $p$ is in the ideal generated by $q$ in $A$. It therefore follows, by \cite[Lemma 1.1]{Zhang:rr0_III}, that there are $\ell \in \bbN$ and orthogonal projections $s_0, \dots, s_{\ell-1} \in pAp$ such that
\begin{equation} \label{eq:si}
p = \sum\nolimits_{h < \ell} s_h \quad   s_0 \preceq s_1 \preceq \dots \preceq s_{\ell-1} \preceq q.
\end{equation}

Without loss of generality, we can assume that $\ell = n \ell_0$ for some $\ell_0 \in \bbN$---add otherwise a finite number of projections equal to zero at the beginning of the sequence $s_0, \dots, s_{\ell -1}$. Define next
\begin{equation}
r_i \coloneqq \sum\nolimits_{i \equiv k \ \mathrm{mod}(n)}  s_k, \quad i < n.
\end{equation}
It follows by \eqref{eq:si} that
\begin{equation}
p = \sum\nolimits_{i < n}r_i, \quad r_0 \preceq r_1 \preceq \dots \preceq r_{n-1} \preceq r_0 + s_{n \ell_0 - 1}.
\end{equation}
These relations imply, using $s_{n\ell_0 -1} \preceq q$, that
\begin{equation}
\tau(r_0) \le \frac1n \tau(p) \le \tau(r_{n-1}) \stackrel{\eqref{eq:small_q}}{<} \tau(r_0) + \frac{\e}{2n}, \quad \tau \in T(A),
\end{equation}
hence in particular
\begin{equation} \label{eq:1n_taup}
\left| \tau(r_0) - \frac1n \tau(p) \right| < \frac{\e}{2n}, \quad \tau \in T(A).
\end{equation}

For every $i < n$, let $p_i$ be a subprojection of $r_i$ equivalent to $r_0$. The  projections $p_0, \dots, p_{n-1} \in pAp \subseteq \overline{aAa}$ are pairwise orthogonal and equivalent, and moreover staisfy the inequalities below
\begin{equation}
\tau\left( \sum\nolimits_{i < n} p_i \right) \stackrel{\eqref{eq:1n_taup}}{>} \tau(p) - \frac{\e}{2} \stackrel{\eqref{eq:trace_p}}{\ge} \tau(a) - \e, \quad \tau \in T(A).
\end{equation}
This proves the lemma and shows that $a$ is tracially almost divisible in $A$, by \Cref{remark:matrix}.
\end{proof}

\begin{proposition} \label{prop:rr0_tad}
Let $A$ be a simple, unital, non-elementary $C^\ast$-algebra such that $T(A) \ne \emptyset$. If $A^\cU$ has real rank zero, then $A$ is tracially almost divisible.
\end{proposition}

\begin{proof}
As $M_m(A)^\cU \cong M_m(A^\cU)$, if $A^\cU$ has real rank zero, then so does $M_m(A)^\cU$, for every $m \in \bbN$, by \cite[Theorem 2.10]{BP:rr0}. Because of this, in order to show that $A$ is tracially almost divisible, it suffices to show that elements in $A_{1,+}$ are almost divisible.

Let thus $n \ge 1$, $a \in A_{1,+}$ and fix $\e > 0$. By simplicity and non-elementarity of $A$, for every $\delta > 0$ there is $b \in A_+$ such that $d_\tau(b) < \delta$ for all $\tau \in T(A)$---this follows for instance by Glimm's theorem \cite[Proposition 3.10]{RobRor:divisibility}---so in particular, identifying $b$ with the corresponding constant sequence in $A^\cU$, we get
\begin{equation}
d_\tau(b) < \delta, \quad \tau \in T_\cU(A).
\end{equation}
By simplicity of $A$, the ideal generated by $b$ in $A^\cU$ contains $1_A = 1_{A^\cU}$, hence it is the whole $A^\cU$, hence  by \Cref{lemma:rr0_tad} there is a $*$-homomorphism
\begin{equation}
\Phi \colon M_n \to \overline{a A^\cU a}
\end{equation}
such that
\begin{equation} \label{eq:trace_Phi}
\tau(\Phi(1_n)) > \tau(a) - \e, \quad \tau \in T_\cU(A).
\end{equation}

From here, it is possible to lift $\Phi$ to a c.p.c. order zero map $\phi \colon M_n \to \overline{aAa}$ such that $\tau(\phi(1_n)) > \tau(a) - \e$ with an argument analogous to that of \ref{item1:gamma_tad} $\Rightarrow$ \ref{item2:gamma_tad} in \Cref{prop:gamma_tad}.
\end{proof}

We can use \Cref{prop:rr0_tad} to give a partial answer to \cite[Problem XXI]{STW:99problems}, which asks whether simple, nuclear $C^\ast$-algebras whose tracial ultrapowers has real rank zero have uniform property $\Gamma$.

\begin{theorem} \label{thm:rr0_gamma}
Let $A$ be a separable, simple, unital, non-elementary $C^\ast$-algebra with $T(A) \ne \emptyset$ and with tracially locally finite nuclear dimension. If $A^\cU$ has real rank zero, then $A$ is uniformly McDuff, and in particular it has uniform property $\Gamma$
\end{theorem}
\begin{proof}
The $C^\ast$-algebra $A$ is tracially almost divisible by \Cref{prop:rr0_tad}, hence it is uniformly McDuff by \Cref{prop:gamma_tad}. Uniform property $\Gamma$ follows by \cite[Theorem 4.6]{CETW:gamma}.
\end{proof}

We can finally deduce \Cref{thm_main:sr1_gamma}.

\begin{corollary} \label{cor:sr1_tad_gamma}
Let $A$ be a separable, simple, unital, non-elementary $C^\ast$-algebra with stable rank one and such that $T(A) \ne \emptyset$. Then $A$ is tracially almost divisible. If moreover $A$ has tracially locally finite nuclear dimension, then $A$ is uniformly McDuff, and in particular it has uniform property $\Gamma$.
\end{corollary}
\begin{proof}
The tracial ultrapower $A^\cU$ has stable rank one by \cite[Theorem A]{Fu:sr1_rr0}, hence $A$ is uniformly McDuff, and thus has uniform property $\Gamma$, by \Cref{thm:rr0_gamma}.
\end{proof}

\section{Tracially LTH-algebras with locally flat dimension growth} \label{S:LTH}
In this section we present a route alternative to \cite{Fu:sr1_rr0} to prove real rank zero of tracial ultrapowers (\Cref{thm:LTH_rr0}). Instead of using stable rank one, we concentrate on a class of $C^\ast$-algebras that satisfy the following form of tracial local homogeneity.

\begin{definition} \label{def:LTH}
Let $A$ be a $C^\ast$-algebra such that $T(A) \ne \emptyset$. Then $A$ is a \emph{tracially locally trivally homogeneous $C^\ast$-algebra} (a \emph{tracially LTH-algebra} for short) if for every finite subset of contractions $F \subset A_1$ and every $\e > 0$ there exists a $C^\ast$-subalgebra $C \subseteq A$ with
\begin{equation} \label{eq:LTH}
C \cong \bigoplus_{i < m} M_{n_i}(C_0(Z_i)),
\end{equation}
where $m, n_0, \dots, n_{m-1} \in \bbN$ and each $Z_i$ is a locally compact metric space, such that for every $a \in F$ there is some $c \in C_1$ satisfying $\| a- c \|_{2, T(A)} < \e$.

We say that a tracially LTH-algebra $A$ has \emph{locally flat dimension growth} if  for every finite subset  $F \subset A_1$ and $\e > 0$ there is a constant $N > 0$ such that, for every $k \in \bbN$ , the approximating $C^\ast$-subalgebra $C \cong \bigoplus_{i < m} M_{n_i}(C_0(Z_i))$ in \eqref{eq:LTH} can be chosen so that $n_i \ge k$ for all $i<m$ and
\begin{equation} \label{eq:locally flat_dim}
\max_{i < m} \left\{\frac{\dim(Z_i)}{n_i} \right\} < N.
\end{equation}
\end{definition}

Locally flat dimension growth is inspired by the almost homonymous condition for AH-algebras introduced in \cite{toms:growth}. Compared to the last condition, where the dimension growth ratio is globally controlled by a single constant, the \emph{local} adjective  in \Cref{def:LTH} emphasizes that the constant $N$ can vary depending on the subset we wish to approximate and the degree of tolerance of the approximation. This detail makes the property extremely versatile and easy to verify in concrete examples, as we will see for instance in \Cref{prop:LTH_crossed}.

The main goal of this section is showing that if $A$ is a simple, separable, unital, non-elementary, tracially LTH-algebra with locally flat dimension growth then $A^\cU$ has real rank zero. This will ensure tracial almost divisibility of $A$ by \Cref{prop:rr0_tad} and in fact uniform property $\Gamma$ itself, since the local structure of tracially LTH-algebras immediately implies tracially locally finite nuclear dimension.

\Cref{thm:LTH_rr0} follows by an argument of topological flavor which generalizes an idea sketched in \cite[Example 3.5]{EN:propS}. This proof exploits an equivalent characterization of $A^\cU$ having real rank zero introduced in \cite[Definition 3.1]{EN:propS} and named \emph{property (S)}. This property (S), not to be confused with the one in \cite{EPR}, originates as a  $C^\ast$-algebraic adaptation of the small boundary property for topological dynamical systems. We report it in the proposition below, after introducing some notation.

Given $0<\delta < 1$, we let $\eta_\delta$ be the continuous function $\eta_\delta \colon [-1,1] \to [0,1]$ defined as
\begin{equation}
\eta_\delta(x) \coloneqq \begin{cases} 1 & \text{if } x \in [-\delta/2, \delta/2], \\ 0 & \text{if } x \in [-1, -\delta] \cup [\delta, 1], \\ \text{linear } & \text{elsewhere}
\end{cases}
\end{equation}

\begin{proposition}[{\cite[Proposition 3.3]{EN:propS}}] \label{prop:tr_rr0}
Let $A$ be a unital $C^\ast$-algebra such that $T(A) \ne \emptyset$. The following are equivalent.\footnote{The original definition of property (S) in \cite[Definition 3.1]{EN:propS} does not have the additional condition that both $a$ and $b$ are contractions. This is however implicitly used in the proof of \cite[Proposition 3.3]{EN:propS}.}
\begin{enumerate}
\item $A^\cU$ has real rank zero.
\item For every $a \in A_{1,\sa}$ and $\epsilon > 0$ there are $b \in A_{1,\sa}$ and $0< \delta < 1$ such that
\begin{enumerate}[ref=\alph*]
\item $\| a - b \|_{2, T(A)} < \e$,
\item \label{itemb:propS} $\tau(\eta_\delta(b)) < \e$ for all $\tau \in T(A)$.
\end{enumerate}
\end{enumerate}
\end{proposition}

\subsection{A topological intermezzo}
The following subsection sets the stage for the proof of \Cref{thm:LTH_rr0} and it is based on a series of classical results from dimension theory and differential topology, with the main one being Sard's theorem for smooth manifolds. We refer the reader to \cite{engelking:dim_fininf, GP:diff_top, Pears:dim} for all the necessary background.

We recall that, given a topological space $X$ and a family $\cU$ of open subsets of $X$, the \emph{order of $\cU$}, denoted $\mathrm{ord}(\cU)$, is the smallest integer $n \ge -1$ such that any $n+2$ tuple of elements in $\cU$ has empty intersection. The notation $\dim(X)$ denotes the \emph{covering dimension} of the topological space $X$, that is the smallest integer $n \in \bbN$ such that any finite open cover of $X$ has a finite open refinement of order at most $n$, or $\infty$ if no such integer exists.

\begin{proposition} \label{prop:general_position}
Fix $n \in \bbN$. Let $Z$ be a compact metric space, let $V \subseteq \bbR^n$ be a finite union of smooth manifolds  and let $f \colon Z \to \mathbb{R}^n$ be a continuous function. Suppose that
\begin{equation}
\dim(Z) + \dim(V) < n.
\end{equation}
Then, for every $\epsilon > 0$, there exists a continuous $g \colon Z \to \mathbb{R}^n \setminus V$ such that $\| f(z) - g(z) \|_2 < \epsilon$ for all $z \in Z$, where $\| \cdot \|_2$ is the Euclidean norm on $\bbR^n$.
\end{proposition}
\begin{proof}
Set $d \coloneqq \dim(Z)$ and, using uniform continuity of $f$, find an open cover $\cU$ of $X$ with $\mathrm{ord}(\cU) \le d$ and such that
\begin{equation}
\| f(x) - f(y)\|_2 < \frac{\e}{3}, \quad x, y \in U, \, U \in \cU.
\end{equation}

By \cite[Theorem 1.10.2]{engelking:dim_fininf} there is a collection of points $(p_U)_{U \in \cU} \subseteq \bbR^n$ which are in general position---that is if $p_0, \dots, p_{n-1} \in (p_U)_{U \in \cU}$ are distinct, then they are affinely independent---and such that the diameter of the set $f(U) \cup \{p_U\}$ is smaller than $\e/2$, for each $U \in \cU$. Let $N(\cU)$ be the \emph{nerve} $N(\cU)$ of $\cU$ with vertices $(p_U)_{U \in \cU}$, that is $N(\cU)$ is the simplicial complex which contains the simplex $\mathrm{conv}(p_{U_0}, \dots, p_{U_{m-1}}) $ if and only $U_0 \cap \dots \cap U_{m-1} \ne \emptyset$, whenever $U_0, \dots, U_m \in \cU$

Let next $P \subseteq \bbR^n$ be the polyhedron underlying $N(\cU)$ and fix a partition of the unity $(h_U)_{U \in \cU}$ of $Z$ subordinate to $\cU$. By \cite[Theorem 1.10.6-7]{engelking:dim_fininf}, the map $f_0 \colon Z \to \bbR^n$ sending $z \mapsto \sum_{U \in \cU} h_U(z) p_U$ is such that
\begin{equation} \label{eq:f_approx}
\| f(z) - f_0(z) \|_2 < \e/2, \quad z \in Z,
\end{equation}
and that $f_0(Z) \subseteq P$.

Notice that $\mathrm{ord}(\cU) \le d $ automatically implies that all simplices comprising $P$ have dimension at most $d$. Furthermore, $P$ is equal to the union of the interiors of all simplices in $N(\cU)$, hence $P = \bigcup_{i < s} P_i$, where each $P_i$ is a smooth manifold whose dimension is at most $d$. By assumption we moreover have that $V = \bigcup_{j < t} V_j$, where each $V_j$ is a smooth manifold with dimension at most $\dim(V)$.

Consider then, for $i <s$ and $j < t$, the map $F_{i,j} \colon P_i \times V_j \to \bbR^n$ where $F_{i,j}(x,y) \coloneqq x - y$. Since
\begin{equation}
\dim(P_i \times V_j)\le \dim(X) + \dim(V) < n, \quad i < s, j < t,
\end{equation}
by Sard's theorem \cite[p. 40]{GP:diff_top} the set
\begin{equation}
P - V \coloneqq \{ p - v : p \in P, \, v \in V \} = \bigcup\nolimits_{i,j}F_{i,j}(P_i \times V_j),
\end{equation}
has Lebesgue measure zero, and thus empty interior. This means that it is possible to find some $w \in \bbR^n \setminus (P - V)$ such that $\| w \|_2 < \e/2$.

Define finally
\begin{equation}
\begin{split}
g \colon Z &\to \bbR^n \\
z &\mapsto f_0(z) - w
\end{split}
\end{equation}
The choice of $w$, along with \eqref{eq:f_approx}, yields $\| f(z) - g(z) \|_2 < \e$ and $g(z) \notin V$, for all $z \in Z$.
\end{proof}

In the following lemma we use \Cref{prop:general_position} to approximate elements of $M_n(C_0(Z))$, for some locally compact $Z$ with sufficiently small dimension, with elements satisfying condition \ref{itemb:propS} in the second item of \Cref{prop:tr_rr0} up to a prescribed degree of tolerance.

\begin{lemma} \label{lemma:approx_rank}
Let $Z$ be a locally compact metric space with $\dim(Z) \le d < \infty$, let $T \subseteq T(M_n(C_0(Z)))$ be a compact subset of traces and suppose that $k,n \in \bbN$ are such that $k \le n$ and $d < k^2$. Then, for every $\e > 0$ and every contraction $f \in M_n(C_0(Z))_{1,\sa}$, there exists a contraction $g \in M_n(C_0(Z))_{1,\sa}$ and some $\delta > 0$ such that
\begin{enumerate}
\item $\| f - g \| < \e$,
\item \label{item2:small_trace} $\tau(\eta_\delta(g)) < \frac{k}{n}$, for all $\tau \in T$.
\end{enumerate}
\end{lemma}
\begin{proof}
We treat first the case of $Z$ compact, in which case we take $T = T(M_n(C(Z)))$. Let $f \in M_n(C(Z))_{1,\sa}$ be a contraction and fix $\e > 0$. To prove the lemma, it suffices to find a continuous function $g \colon Z \to M_n(\bbC)_{1,\sa}$ such that
\begin{enumerate}[label=(\roman*)]
\item $\| f - g \| < \e$,
\item $\rank(g(z)) > n - k$, for all $z \in Z$.
\end{enumerate}

Indeed, if $g$ is as above and $z \in Z$, then $\rank(g(z)) > n-k$ implies that 0 is an eigenvalue of multiplicity less than $k$ for $g(z)$, and in particular that there is $\delta_z > 0$ such that $\mathrm{tr}_n(\eta_{\delta_z}(g)(z)) < k/n$, where $\mathrm{tr}_n$ is the normalized trace on $M_n(\bbC)$. By continuity of $\eta_{\delta_z}(g)$ it is then possible to find an open neighborhood $U_z$ of $z$ such that
\begin{equation}
\mathrm{tr}_n(\eta_{\delta_z}(g)(y)) < k/n, \quad y \in U_z.
\end{equation}
Compactness of $Z$ then yields some $\delta > 0$ such that
\begin{equation}\label{eq:uniform_delta}
\mathrm{tr}_n(\eta_\delta(g)(z)) < k/n, \quad z \in Z.
\end{equation}

If we finally pick some $\tau \in T(M_n(C(Z)))$, then $\tau(\eta_\delta(g)) < {k/n}$ follows by \eqref{eq:uniform_delta} since $\tau$ is always equal to $\mu \otimes \mathrm{tr}_n$, for some probability Radon measure $\mu$ on $Z$.

We can then focus on finding a continuous function $g$ approximating $f$ and whose images always have rank greater than $n-k$.

Consider the set
\begin{equation}
V \coloneqq \{ a \in M_n(\bbC)_\sa : \rank(a) \le n-k \}.
\end{equation} 
From here on, we identify $M_n(\bbC)_\sa$ with $\mathbb{R}^{n^2}$. For every $\ell \le n$, the set of all self-adjoint matrices of rank $\ell$ is a smooth submanifold of $M_n(\bbC)_\sa$ of dimension $\ell(2n - \ell)$ (see e.g. \cite[Exercise 13, p. 27]{GP:diff_top}), hence the dimension of $V$ is at most $n^2 - k^2$.

Since $d < k^2$, it follows that
\begin{equation}
\dim(Z) + \dim(V) = d + n^2 - k^2  < n^2 = \dim(M_n(\bbC)_\sa),
\end{equation}
hence, as the Frobenius and the $C^\ast$-norm are equivalent on $M_n(\bbC)_\sa \cong \mathbb{R}^{n^2}$, by \Cref{prop:general_position} there is a continuous function $g_0 \colon Z \to M_n(\bbC)_\sa \setminus V$ such that $\| f - g_0 \| < \e/2$.

We finally normalize $g_0$ and take $g \coloneqq  g_0/\| g_0 \|$, which gives the desired approximation of $f$.

Suppose next that $Z$ is non-compact and locally compact, and let $\alpha Z \coloneqq Z \cup \{ \infty \}$ be the one-point compactification of $Z$---we still have $\dim(\alpha Z) \le d$, see e.g. \cite[Proposition 3.5.6]{Pears:dim}. Fix $f \in M_n(C_0(Z))_{1,\sa}$, some $\e >0$ and identify $M_n(C_0(Z))$ with the ideal of all functions in $M_n(C(\alpha Z))$ which annihilate at $\infty$.

By the previous part of the proof, there is $h \in M_n(C(\alpha Z))_{1,\sa}$ and $\delta > 0$ such that
\begin{equation} \label{eq:first_approx}
\| f - h \| < \e/2 \quad \text{and} \quad \tau(\eta_\delta(h)) < k/n, \quad \tau \in T(M_n(C(\alpha Z))).
\end{equation}

We claim next that there exists $\gamma > 0$ sufficiently small such  that if $h_0 \in M_n(C(\alpha Z))_{1,\sa}$, then
\begin{equation} \label{eq:2norm_approx}
\| h - h_0 \|_{2, \tau} < \gamma \Rightarrow  \tau(\eta_\delta(h_0)) < \frac{k}{n} , \quad \tau \in T(M_n(C(\alpha Z))).
\end{equation}
This follows using for instance \cite[Lemma 3.1]{GGKNV:tr_amen}---whose proof works verbatim for self-adjoint contractions---since each trace $\tau$ is 1-Lipschitz for the norm $\| \cdot \|_{2,\tau}$.

Let then $(e_i)_{i=1}^\infty$ be an approximate unit of $M_n(C_0(Z))$ and, using compactness of $T$, let $i \in \bbN$ be big enough that
\begin{enumerate}[label=(\alph*)]
\item \label{item1:en} $\| f e_i - f \| < \e/2$,
\item \label{item2:en} $\| 1 - e_i \|_{2, \tau} < \gamma$, for all $\tau\in T$.
\end{enumerate}

Consider finally $g \coloneqq he_i \in M_n(C_0(Z))_{1,\sa}$. We have that
\begin{equation}
\| f -g \| \le \| f - fe_i \| + \| fe_i - he_i\| \stackrel{\eqref{eq:first_approx}}{<} \frac{\e}{2} + \frac{\e}{2} = \e.
\end{equation}
To verify that $g$ also satisfies condition \ref{item2:small_trace} of the statement of the lemma, note that $T \subseteq T(M_n(C_0(Z)))$ can be naturally identified with a subset of $T(M_n(C(\alpha Z)))$, since every Radon probability measure on $Z$ canonically extends to a Radon probability measure on $\alpha Z$. We therefore have
\begin{equation}
\| h -g \|_{2,\tau} \le \|h \| \cdot \| 1 - e_i \|_{2, \tau} \stackrel{\mathrm{\ref{item2:en}}}{<} \gamma, \quad \tau \in T,
\end{equation}
hence \eqref{eq:2norm_approx} yields $\tau(\eta_\delta(g)) < k/n$ for all $\tau \in T$.
\end{proof}

\subsection{Tracially LTH-algebras and uniform property $\mathbf{\Gamma}$}
We prove next that the tracial ultrapower of a simple  tracially LTH-algebra with locally flat dimension growth has real rank zero.
\begin{theorem} \label{thm:LTH_rr0}
Let $A$ be a separable, simple, unital, non-elementary, tracially LTH-algebra with locally flat dimension growth. Then $A^\cU$ has real rank zero.
\end{theorem}
\begin{proof}
By \Cref{prop:tr_rr0} it suffices to find, given $a \in A_{1,\sa}$ and $\e > 0$, some $b \in A_{1,\sa}$ and $\delta >0$ such that $\| a- b\|_{2, T(A)} < \e$ and that
\begin{equation}
\tau(\eta_\delta(b)) < \e, \quad \tau \in T(A).
\end{equation}

Since $A$ is a tracially LTH-algebra with locally flat dimension growth, there are $N > 0$, a $C^\ast$-subalgebra $C \subseteq A$ with
\begin{equation}
C \cong \bigoplus_{i < m} M_{n_i}(C_0(Z_i)),
\end{equation}
such that 
\begin{equation} \label{eq:dim_ratio}
\max_{i < m} \left\{\frac{\dim(Z_i)}{n_i} \right\} < N,
\end{equation}
and there is some $c \in C_1$ such that $\| a - c \|_{2, T(A)} < \e/2$. We moreover assume that each $n_i$ is so large that
\begin{equation} \label{eq:n2_e}
(N + 2\e )n_i+ 1 < \e^2 n_i^2, \quad i < m.
\end{equation}

Write $c$ as $c_0 \oplus \dots  \oplus c_{m-1}$, with each $c_i \in M_{n_i}(C_0(Z_i))_{1,\sa}$. For each $i < m$, let $S_i$ be the set of tracial functionals on $M_{n_i}(C_0(Z_i))$ that are restrictions of traces on $A$. This set is compact since $T(A)$ is, and it does not contain the zero functional since $A$ is simple and non-elementary. It then follows that the subset of $T(M_{n_i}(C_0(Z_i)))$ obtained by normalizing the elements of $S_i$
\begin{equation}
T_i \coloneqq \{ \tau/\| \tau \| : \tau \in S_i \}.
\end{equation}
is compact. Set moreover $M_i \coloneqq \max_{\tau \in S_i} \| \tau\|$ and $\e_i \coloneqq \e /(mM_i)$.

For every $i < m$, find $k_i \le n_i$ such that $k_i/n_i \le \e_i \le (k_i+1)/n_i$. We then have
\begin{equation}
\begin{split}
\dim(Z_i) &\stackrel{\eqref{eq:dim_ratio}}{<} n_i N \\
& \stackrel{\eqref{eq:n2_e}}{<} \e_i^2 n_i^2  - 2 \e_i n_i - 1 \\
&\stackrel{\hphantom{\eqref{eq:n2_e}}}{<} (k_i+1)^2 -2k_i -1 \\
&\stackrel{\hphantom{\eqref{eq:n2_e}}}{=} k_i^2.
\end{split}
\end{equation}

\Cref{lemma:approx_rank} then gives $\delta_i > 0$ and contractions $b_i \in M_{n_i}(C_0(Z_i))_{1,\sa}$ such that $\| c_i - b_i \| < \e /2$ and that
\begin{equation}
\tau(\eta_{\delta_i}(b_i)) < \e_i, \quad \tau \in T_i, \ i < m.
\end{equation}

Then $b \coloneqq b_0 \oplus \dots b_{m-1}$ and $\delta \coloneqq \min_{i < m} \delta_i$ are as required.
\end{proof}

We finally deduce uniform property $\Gamma$ for simple tracially LTH-algebras with locally flat dimension growth.
\begin{theorem} \label{thm:LTH_gamma}
Let $A$ be a separable, simple, unital, non-elementary, tracially LTH-algebra with locally flat dimension growth. Then $A$ is uniformly McDuff, and in particular it has uniform property $\Gamma$.
\end{theorem}
\begin{proof}
Tracial almost divisibility of $A$ is a consequence of \Cref{thm:LTH_rr0} and \Cref{prop:rr0_tad}. Tracially LTH-algebras have tracially locally finite nuclear dimension since commutative $C^\ast$-algebras whose spectra have finite covering dimension have finite nuclear dimension by \cite[Proposition 2.4]{WZ:nuc_dim}, and since finite nuclear dimension is preserved by taking matrix amplifications and direct sums \cite[Propositions 2.3, 2.5]{WZ:nuc_dim}.
\end{proof}

\begin{remark}
A natural generalization of \Cref{def:LTH}---and of the second part of \Cref{def:loc_fin_nucdim}---would be the one allowing $C$ to be a subalgebra of the tracial completion $\overline{A}^{T(A)}$ of $A$ (see \cite[Definition 3.19]{CCEGSTW}) rather than of $A$ itself. The proof of \Cref{thm:LTH_rr0} holds verbatim, and it seems probable that \Cref{thm:LTH_gamma} would generalize as well, after adapting---most likely without difficulties---the arguments of \cite[\S 5]{Winter:pure} to the setting of tracial ultrapowers and tracial completions.
\end{remark}

\section{Uniform property $\Gamma$ for AH-algebras and crossed products} \label{S:LTH_examples}
In this section we verify uniform property $\Gamma$ on various classes of AH-algebras and of crossed products of free minimal topological dynamical systems, including those featuring in \Cref{thm_main:Vill_cp}.

\subsection{Diagonal AH-algebras and AH-algebras with flat dimension growth} \label{ss:AH}
We recall that a $C^\ast$-algebra is \emph{approximately homogeneous} (or, simply, an \emph{AH-algebra}) if $A$ is an inductive limit $\lim_{s \to \infty} (A_s, \Phi_s)$ where each $A_s$ is a $C^\ast$-algebra of the form
\begin{equation} \label{eq:AH}
A_s = \bigoplus_{i < m_s} p_{i,s} M_{n_{i,s}}(C(Z_{i,s}))p_{i,s}, \quad s \in \bbN,
\end{equation}
for some natural numbers $m_s, n_{i,s}$, some compact metric spaces $Z_{i,s}$ and some projections $p_{i,s} \in M_{n_{i,s}}(C(Z_{i,s}))$. 

\emph{Diagonal} AH-algebras are obtained by imposing some restraints on the connecting maps $\Phi_s \colon A_s \to A_{s+1}$.
More precisely, given two compact Hausdorff spaces $Y,Z$ and $n,m \in \bbN$, a $*$-homomorphism $\Phi \colon M_n(C(Y)) \to M_{nm}(C(Z))$ is \emph{diagonal} if there are $m$ continuous maps $\lambda_i \colon Z \to Y$ such that
\begin{equation} \label{eq:diagonal}
\Phi(f) \coloneqq \begin{pmatrix}
f\circ \lambda_0 & 0 & \cdots & 0 \\
0 & f \circ \lambda_1 & \cdots & 0 \\
\vdots & \vdots & \ddots & \vdots \\
0 & 0 & \cdots & f\circ \lambda_{m-1}
\end{pmatrix}, \quad f \in M_n(C(Y)).
\end{equation}
More generally, a unital $*$-homomorphism
\begin{equation}
\Phi \colon \bigoplus_{i < s} M_{n_i}(C(Y_i)) \to M_{m}(C(Z))
\end{equation}
is \emph{diagonal} if it is direct sum of diagonal $*$-homomorphisms in the sense of \eqref{eq:diagonal} and, finally, a unital $*$-homomorphism
\begin{equation}
\Phi \colon \bigoplus_{i < s} M_{n_i}(C(Y_i)) \to \bigoplus_{j < t} M_{m_j}(C(Z_j))
\end{equation}
is \emph{diagonal} if each co-restriction $\Phi \colon \bigoplus_{i < s} M_{n_i}(C(Y_i)) \to M_{m_j}(C(Z_j))$ is diagonal. We finally say that an AH-algebra $A \cong \lim_{s \to \infty} (A_s, \Phi_s)$ is \emph{diagonal} if $A_s \cong   \bigoplus_{i < s} M_{n_i}(C(Z_i))$ and $\Phi_s$ is a diagonal unital $*$-homomorphism, for every $s \in \bbN$ (see \cite[\S 2.2]{EHT:sr1} for details).

Villadsen algebras of the first type constitute a subclass of diagonal AH-algebras with various restrictions including that $A_s \cong M_{n_s}(C(Z^{d_s}))$ and that $\Phi_s$ is a diagonal $*$-homomorphism where each $\lambda_i\colon Z^{n_{s+1}} \to Z^{n_s}$ as in \eqref{eq:diagonal} is either a coordinate projection or a constant map (see \cite{Villadsen:1} or \cite[\S 3.1]{TW:1} for details).

The following corollary recovers the first part of \Cref{thm_main:Vill_cp}.
\begin{corollary} \label{cor:vill}
Let $A$ be a simple, non-elementary, diagonal AH-algebra or, more generally, a simple, non-elementary AH-algebra with stable rank one. Than $A$ has uniform property $\Gamma$.
\end{corollary}
\begin{proof}
By \cite{EHT:sr1}, simple, unital, diagonal AH-algebras have stable rank one. Since locally finite nuclear dimension is a direct consequence of approximate homogeneity, the following result immediately follows from \Cref{cor:sr1_tad_gamma}.
\end{proof}


\begin{remark} \label{remark:Vill}
Note that \Cref{cor:vill} can be partially recovered using  \Cref{thm:LTH_gamma}. This is indeed the case if  $A \cong \lim_{s \to \infty} (A_s, \Phi_s)$ is a simple, non-elementary, unital, diagonal AH-algebra such that $A_s \cong M_{n_s}(C(Z_s))$ and whose connecting maps are unital and injective---e.g. a Villadsen algebra of the first type. To see this, notice first that $\lim_{s \to \infty} n_s = \infty$, by simplicity and non-elementarity. Then, given a finite subset $F \subset A_1$ and $\e > 0$, there exists $s \in \bbN$ such that the elements of $F$ can be approximated, up to $\e$, by elements of $A_s$. Since $F$ is finite, one only needs finitely elements $f_0, \dots, f_{m-1} \in C(Z_s)$ to approximate elements of $F$ as sums of elementary tensors in $A_s \cong C(Z_s) \otimes M_{n_s}$. We can thus assume that the approximation is done in a $C^\ast$-subalgebra $M_{n_s}(C(Y_s)) \subseteq A_s$ such that $Y_s$ has finite covering dimension. For $t > s$, let $\Phi_{t,s} \coloneqq \Phi_t \circ \Phi_{t-1} \circ \dots \circ \Phi_s$ be the diagonal connecting map from $A_s$ to $A_t$. Then each $\Phi_{t,s}(f_i)$ can be written as sum of $n_t/n_s$ elementary tensors in $A_t \cong C(Z_t) \otimes M_{n_t}$, which means that $\Phi_{t,s}(M_{n_s}(C(Y_s)))$ is mapped into a unital $C^\ast$-subalgebra of $A_t$ isomorphic to $M_{n_t}(C(Y_t))$, for a space $Y_t$ such that $\dim (Y_t) \le \dim (Y_s) n_t/n_s$. In particular
\begin{equation}
\frac{\dim (Y_t)}{n_t} \le \frac{\dim(Y_s)}{n_s}, \quad t > s.
\end{equation}
Since, $n_t$ can be chosen as large as desired, it follows that $A$ is a tracially LTH-algebra with locally flat dimension growth.
\end{remark}

We remark that it is not clear to us whether \emph{all} simple, non-elementary AH-algebras are tracially LTH-algebra with locally flat dimension growth. However, AH-algebras that have  \emph{flat dimension growth} in the sense of \cite{toms:growth}, clearly do, hence the following is an immediate corollary of  \Cref{thm:LTH_gamma}. Note that these AH-algebras need not have, a priori, stable rank one.

\begin{corollary} \label{cor:AH_fdg}
Let $A \cong \lim_{s \to \infty}(A_s, \Phi_s)$ be a simple, unital, non-elementary AH-algebra such that $A_s \cong   \bigoplus_{i < s} M_{n_i}(C(Z_i))$ for every $s \in \bbN$ and such that
\begin{equation}
\lim_{s \to \infty} n_s = \infty, \quad \limsup_{s \to \infty} \max_{i <s} \left\{\frac{\dim(Z_i)}{n_s}\right\} < \infty.
\end{equation}
 Then $A$ has uniform property $\Gamma$.
\end{corollary}

\subsection{Crossed products}
In this subsection we identify a family of free minimal actions by amenable groups whose corresponding crossed products are tracially LTH-algebras with locally flat dimension growth. It turns out that Niu's \emph{uniform Rokhlin property} from \cite{niu:rc_mdim_urp} (\emph{URP} for short), which we recall below in \Cref{def:urp}, is sufficient for this.

Our proof that actions with the URP generate crossed products that are LTH-algebras with locally flat dimension growth is based on the proof of \cite[Theorem 3.9]{niu:rc_mdim_urp}, and is presented in \Cref{prop:LTH_crossed}. Before that, we introduce some notation and definitions.

Let $X$ be a topological space. Given two open covers of $X$, say $\cU$ and $\cV$, we say that $\cV$ \emph{refines} $\cU$ if every $V \in \cV$ is included in some $U \in \cU$. We define next, for an open cover $X$ of $\cU$ the invariant
\begin{equation}
\cD(\cU) \coloneqq \min_{\cV \text{ refines } \cU} \mathrm{ord}(\cV).
\end{equation}
Given some open covers $\cU_0, \dots, \cU_{n-1}$ of $X$, we let $\bigwedge_{i < n} \cU_i$ be the open cover of $X$ whose elements are the intersections $U_0 \cap \dots \cap U_{n-1}$, where $U_i \in \cU_i$  for every $i < m$. Recall that $\cD(\cU)$ is sub-additive (see \cite[Corollary 2.5]{LW:mdim}), that is
\begin{equation}\label{eq:sub-additive}
\cD \left(\bigwedge\nolimits_{i < n} \cU_i\right) \le \sum\nolimits_{i < n} \cD(\cU_i).
\end{equation}

Let $G$ be a (discrete) group. Let $K \subseteq G$ be a finite subset and let $\e > 0$. A subset $F \subseteq G$ is $(K, \e)$-invariant if
\begin{equation}
\frac{| FK \Delta F|}{|F|} < \e.
\end{equation}
A countable group is \emph{amenable} if there is a sequence $(F_n)_{n \in \bbN}$ of finite subsets of $G$ such that $G = \bigcup_{n \in \bbN} F_n$ and that, for every finite $K \subseteq G$ and $\e > 0$, there is $m \in \bbN$ such that $F_n$ is $(K, \e)$-invariant for all $n \ge m$. The sets $F_n$ are called \emph{F{\o}lner sets}, and $(F_n)_{n \in \bbN}$ is a \emph{F{\o}lner sequence}.
Given finite subsets $K,F \subseteq G$, the \emph{$K$-interior of $F$} is defined as
\begin{equation}
\mathrm{int}_K(F) \coloneqq \{g \in F: gK \subseteq F \}.
\end{equation}
Note that $| F \setminus \mathrm{int}_K(F) | \le |K | | FK \Delta F|$, hence, if $F$ is $(K, \e/|K|)$-invariant, then
\begin{equation} \label{eq:interior}
\frac{| F \setminus \mathrm{int}_K(F) | }{|F |} < \e.
\end{equation}

Finally, for an action $G \curvearrowright X$, given $g \in G$, $B \subseteq X$ and $\cU \subseteq \cP(X)$, we denote by $gB$ the set of all points $gx$ for $x \in B$, and by $g \cU$ the family of all $gU$ for $U \in \cU$.

From here on (and everywhere else in the paper), by \emph{action} of a group on a topological space, we mean \emph{continuous action by homeomorphisms}.

Given an action $G \curvearrowright X$, a finite tuple $\{ (S_i, B_i) \}_{i < m}$ is a \emph{castle} if each $B_i \subseteq X$ and each $S_i$ is a finite subset of $G$ such that the family $\{s B_i \}_{s \in S_i, i < m}$ is composed of pairwise disjoint sets. The subsets $B_i$'s are the \emph{bases} and the subsets $S_i$'s are the \emph{shapes} of the castle. An \emph{open castle} $\{ (S_i, B_i) \}_{i < m}$ is a castle such that each $B_i$ is open. We abbreviate the set $\bigsqcup_{s \in S_i} s B_i $  with the notation $S_i B_i$. We let $M_G(X)$ denote the set of all Radon measure on $X$ that are left invariant by the action of $G$.

\begin{definition}[{\cite[Definition 3.1]{niu:rc_mdim_urp}}] \label{def:urp}
Let $G$ be a countable amenable group, let $X$ be a compact metric space and let $G \curvearrowright X$ be an action. Then $G \curvearrowright X$ has the \emph{uniform Rokhlin property} if for every finite subset $K \subseteq G$ and every $\e > 0$ there is an open castle $\{ (S_i, B_i) \}_{i < m}$ such that
\begin{enumerate}
\item each shape $S_i$ is $(K,\e)$-invariant,
\item $\mu(X \setminus \bigsqcup_{i < m} S_i B_i) < \e$ for all $\mu \in M_G(X)$.
\end{enumerate}
\end{definition}

The URP is a weakening of \emph{almost finiteness in measure} from \cite{KS:sbp}, whose definition coincides with \Cref{def:urp} except for the additional requirement that the bases $B_i$ can be chosen with arbitrarily small diameter. This seemingly minor modification turns the URP into a much weaker condition: while free, minimal topologicaly dynamical systems that fail almost finiteness in measure exist---for instance, any action with positive mean dimension---it remains unknown whether any such system fails the URP.

\begin{proposition}[{cf. \cite[Theorem 3.9]{niu:rc_mdim_urp}}] \label{prop:LTH_crossed}
Let $G$ be a countably infinite ame\-nable group, let $X$ be a compact metric space, and let $G \curvearrowright X$ be a free action with the URP. Then $C(X) \rtimes G$ is a tracially LTH-algebra with locally flat dimension growth.
\end{proposition}
\begin{proof}
Denote $C(X) \rtimes G$ by $A$ and fix a finite subset of contractions $F \subset A_1$, $k \in \bbN$ and $\e > 0$. Without loss of generality, we can assume there are finite sets $F_0 \subset C(X)$ and $L \subseteq G$, with  $L = L^{-1}$, such that every $a \in F$ is a finite sum of the form
\begin{equation}
a = \sum\nolimits_{\ell \in L} a_\ell u_\ell, \quad a_\ell \in F_0.
\end{equation}

Since $X$ is compact, there exists a finite open cover  $\cU$ of $X$ such that
\begin{equation} \label{eq:uniform_c}
| f(x) - f(y) | < \frac{\e}{4|L|}, \quad f \in F_0, \ x,y \in U, \ U \in \cU.
\end{equation}
The value witnessing locally flat dimension growth of $C(X) \rtimes G$ will depend on $\mathrm{ord}(\cU)$---the impatient reader can peak at \eqref{eq:dim}. 

Find next a sufficiently large finite subset $K \subseteq G$ containing $L$ and a sufficiently small $\delta$ so that if $S \subseteq G$ is $(K, \delta)$-invariant, then $|S | \ge k$ and
\begin{equation} \label{eq:interior}
\frac{| S \setminus \mathrm{int}_{L}(S) |}{|S|} < \frac{\e}{8}.
\end{equation}

We now build the homogeneous $C^\ast$-subalgebra $C$ of $A$ approximating the elements in $F$. Using that $G \curvearrowright X$ has the URP, let $\{ (S_i, B_i) \}_{i < m}$ be an open castle with $(K, \delta)$-invariant shapes and such that
\begin{equation} \label{eq:small_reminder}
\mu\left(X \setminus \bigsqcup\nolimits_{i < m} S_i B_i \right) < \frac{\e}{8}, \quad \mu \in M_G(X).
\end{equation}
By a standard argument using the Portmanteau theorem and compactness of $M_G(X)$---see e.g. \cite[Proposition 3.4]{KS:sbp}---we can shrink each $B_i$ a little bit so that \eqref{eq:small_reminder} still holds, and find open sets $V_i, W_i$ such that $B_i \subseteq \overline{B_i} \subseteq W_i \subseteq \overline{W_i} \subseteq V_i$ and such that $\{ (S_i, V_i) \}_{i < m}$ is also an open castle.

For each $i < m$, let $\cU_i$ be a refinement of $\bigwedge_{g \in S_i} g^{-1} \cU$ such that
\begin{equation}
\mathrm{ord}(\cU_i) = \cD \left( \bigwedge\nolimits_{g \in S_i} g^{-1} \cU \right), \quad i < m,
\end{equation}
hence we have
\begin{equation} \label{eq:ord}
\mathrm{ord}(\cU_i) \stackrel{\eqref{eq:sub-additive}}{\le} |S_i| \cD(\cU) \le | S_i | \mathrm{ord}(\cU), \quad i < m.
\end{equation}
Consider the open covers of $W_i$ and $V_i$ defined as
\begin{equation}
\cW_i \coloneqq \{ U \cap W_i : U \in \cU_i \}, \ \cV_i \coloneqq \cW_i \cup \{ U \cap (V_i \setminus \overline{B_i}) : U \in \cU_i\}, \quad i < m.
\end{equation}
Note that $\cV_i$ satisfies the following inequalities
\begin{equation} \label{eq:2ord}
\mathrm{ord}(\cV_i) \le 2 \mathrm{ord}(\cU_i) + 1 \stackrel{\eqref{eq:ord}}{\le} 2  |S_i| \mathrm{ord}(\cU) +1 , \quad i < m.
\end{equation}

It is possible to find---see e.g. \cite[Lemma 3.10]{niu:rc_mdim_urp}---continuous functions $(\phi^{(i)}_V)_{V \in \cV_i}$ from $X$ to $[0,1]$ such that
\begin{enumerate}[label=(\roman*)]
\item \label{item1:phiV} $\mathrm{supp}(\phi^{(i)}_V) \subseteq V$ for all $V \in \cV_i$, and $i < m$,
\item \label{item2:phiV} $\left\| \sum_{V \in \cV_i} \phi^{(i)}_V\right\| \le 1$,
\item \label{item3:phiV} $\sum_{V \in \cV_i} \phi^{(i)}_V(x) = 1$ for all $x \in \overline{B_i}$ and $i < m$.
\end{enumerate}

Consider next the $C^\ast$-subalgebra $C \subseteq C(X) \rtimes G$ defined as
\begin{equation}
C \coloneqq \mathrm{C}^\ast\left\{ u^\ast_h\phi^{(i)}_V u_g : V \in \cV_i, g, h \in S_i, i < m \right\}.
\end{equation}
Since $\{ (S_i, V_i) \}_{i < m}$ is an open castle, one readily checks that, given $g_0, h_0 \in S_i$, $V \in \cV_i$, $g_1, h_1 \in S_j$ and $W \in \cV_j$, for some $i,j < m$, if $g_0 \ne h_1$ then
\begin{equation}
u^\ast_{h_0} \phi^{(i)}_V u_{g_0} u_{h_1}^\ast \phi^{(j)}_W u_{g_1} = 0,
\end{equation}
hence by \cite[Lemma 3.12]{niu:rc_mdim_urp} it follows that
\begin{equation}
C \cong \bigoplus_{i < m} M_{| S_i |}\left(\mathrm{C}^\ast \left\{ \phi^{(i)}_V : V \in \cV_i \right\}\right),
\end{equation}
with the map sending $u_h^\ast \phi^{(i)}_V u_g$, for $V \in \cV_i$ and $g,h \in S_i$, to the $|S_i| \times |S_i|$ matrix whose $(h,g)$-entry is equal to $\phi^{(i)}_V$ and whose all other entries are zero.

Fix $i < m$ and let $Z_i$ be the spectrum of $\mathrm{C}^\ast \left\{ \phi^{(i)}_V : V \in \cV_i \right\}$. By \cite{EN:minimal}, $Z_i$ is homeomorphic to $(\prod_{V \in \cV_i} \phi_V)(V_i) \subseteq \bbR^{| \cV_i |}$. Since $(\phi^{(i)}_V)_{V \in \cV_i}$ is subordinate to $\cV_i$, each element of $(\prod_{V \in \cV_i} \phi_V)(V_i)$ can have
at most $\mathrm{ord}(\cV_i) + 1$ coordinate different from zero. We thus conclude that $\dim(Z_i) \le \mathrm{ord}(\cV_i) +1$, which in turn gives
\begin{equation} \label{eq:dim}
\frac{\dim(Z_i)}{|S_i|}\stackrel{\eqref{eq:2ord}}{\le} 2( \mathrm{ord}(\cU) +1).
\end{equation}
In particular, the bound on $\dim(Z_i)/|S_i|$ only depends on $\mathrm{ord}(\cU)$, which in turn only depends on $F$ and $\e$.

We now proceed to find elements in  $C_1$  approximating those in $F$. This is done in two steps, the first one being a slight perturbation---in $C^\ast$-norm---of the elements in $F$. To do so, let $\cV$ be the open cover of $X$ defined as
\begin{equation} \label{eq:cV}
\cV \coloneqq \bigcup\nolimits_{i < m} \{ gV : V \in \cV_i, g \in S_i \} \cup \{ U \cap (X \setminus \bigsqcup\nolimits_{i < m} S_i \overline{W_i}) : U \in \cU \}.
\end{equation}
By \cite[Lemma 3.11]{niu:rc_mdim_urp}, it is possible to extend the family
\begin{equation}
\left(\phi^{(i)}_{V} \circ g^{-1}\right)_{V \in \cV_i, g \in S_i, i < m}
\end{equation}
to a partition of the unity $(\phi_V)_{V \in \cV}$ subordinate to $\cV$.

Fix $a \in F$, which we write in the form $a = \sum_{\ell \in L} a_\ell u_\ell$, with $a_\ell \in C(X)$. Pick, for every $V \in \cV$, some $x_V \in V$, and define
\begin{equation}
a_0 \coloneqq \sum\nolimits_{\ell \in L} \left(\sum\nolimits_{V \in \cV} a_\ell(x_V) \phi_V \right) u_\ell.
\end{equation}
Since, by construction, the cover $\cV$ refines $\cU$---recall that each $\cU_i$, of which the $\cV_i$'s are refinements, is a refinement of $\bigwedge_{g \in S_i} g^{-1} \cU$, hence $g \cU_i$ refines $\cU$ for every $g \in S_i$---by \eqref{eq:uniform_c} it follows that $\| a - a_0 \| \le {\e}/{4}$.
Up to normalizing $a_0$ if necessary, we can assume that $a_0 \in A_1$ and that
\begin{equation} \label{eq:a_a0}
\| a - a_0 \| < \e/2.
\end{equation}

The next step is to approximate $a_0$---this time in tracial norm---with a contraction in $C$. This is done by \emph{compressing} $a_0$ to the support of the castle $\bigsqcup_{i < m} \mathrm{int}_L(S_i) W_i$.
To do so, define
\begin{equation}
p_i \coloneqq \sum\nolimits_{W \in \cW_i} \sum\nolimits_{g \in \mathrm{int}_L(S_i)} u_g^\ast \phi^{(i)}_W u_g, \quad i < m,
\end{equation}
and
\begin{equation}
p \coloneqq \sum\nolimits_{i < m} p_i.
\end{equation}

Let $\ell \in L$. We then have
\begin{equation} \label{eq:pu_h}
u_\ell p = \sum\nolimits_{i < m}  \sum\nolimits_{W \in \cW_i} \sum\nolimits_{g \in \mathrm{int}_L(S_i)}u_{g\ell^{-1}}^\ast \phi_W u_g.
\end{equation}
Notice that, whenever $g \in \mathrm{int}_L(S_i)$, $W \in \cW_i$ and $k$ is any element of $S_i$ then 
\begin{equation}
u_{g\ell^{-1}}^\ast \phi_W u_g = \left(u_{g\ell^{-1}}^\ast \phi_W^{1/2} u_k \right) \left(u_k^\ast  \phi_W^{1/2} u_g \right),
\end{equation}
where both  $u_{g\ell^{-1}}^\ast \phi_W^{1/2} u_k$ and $u_k^\ast  \phi_W^{1/2} u_g$ belong to $C$. We can thus deduce that $u_\ell p \in C$, and therefore $pu_\ell p \in C$, for all $\ell \in L$.

We use this to prove that $pa_0 p \in C$. Indeed, since by the definition in \eqref{eq:cV} every $V \in \cV$ is either included in $\bigsqcup_{i < m} S_i V_i$ or disjoint from $\bigsqcup_{i < m} S_i W_i$, and since $p$ is supported on $\bigsqcup_{i < m} S_iW_i$, we have that
\begin{equation}
\begin{split}
pa_0p &= p \sum\nolimits_{\ell \in L} \left(\sum\nolimits_{V \in \cV} a_\ell(x_V) \phi_V \right) u_\ell p \\
&= \sum\nolimits_{\ell \in L} \sum\nolimits_{i <m} \sum\nolimits_{V \in \cV_i} \sum\nolimits_{g \in S_i} a_\ell(x_{gV}) u_g^\ast \phi_V u_g\cdot pu_\ell p,
\end{split}
\end{equation}
which belongs to $C$, being the sum of products of elements in $C$.

To conclude, notice first that $pa_0p$ is a contraction since $p$ is, by \ref{item2:phiV}. We finally check that $\| a - p a_0 p \|_{2, T(A)} < \e$. By \ref{item3:phiV} we have $p(x) = 1$ for all $x \in \bigsqcup_{i < m} \mathrm{int}_L(S_i) B_i$.
This entails
\begin{equation}
\begin{split}
\mu(X \setminus p^{-1}(1)) &\stackrel{\hphantom{\eqref{eq:interior}}}{\le} \mu\left(X \setminus \bigsqcup\nolimits_{i < m}\mathrm{int}_L(S_i)B_i\right) \\
& \stackrel{\hphantom{\eqref{eq:interior}}}{\le} \mu\left(X \setminus \bigsqcup\nolimits_{i < m} S_i B_i\right) + \frac{|S_i \setminus \mathrm{int}_L(S_i)|}{|S_i|} \\
& \stackrel{\eqref{eq:interior}}{<} \mu\left(X \setminus \bigsqcup\nolimits_{i < m} S_i B_i\right) + \frac{\e}{8} \\
& \stackrel{\eqref{eq:small_reminder}}{<} \frac{\e}{4}.
\end{split}
\end{equation}
Since the restrictions of traces on $A\coloneqq C(X) \rtimes G$ to $C(X)$ correspond precisely to the measures in $M_G(X)$, it follows that
\begin{equation} \label{eq:1_p}
\| 1 - p \|_{2, T(A)} < \e/4,
\end{equation}
hence
\begin{equation}
\| a_0 - pa_0p \|_{2,T(A)} \le \| a - a_0 \|_{2,T(A)} + \| a_0 - pa_0 p \|_{2, T(A)} \stackrel{\eqref{eq:a_a0}, \eqref{eq:1_p}}{<} \e.
\end{equation}
\end{proof}

We finally obtain the following.
\begin{corollary} \label{cor:cp}
Let $G$ be a countably infinite amenable group, let $X$ be a compact metric space, and let $G \curvearrowright X$ be a free minimal action with the URP. Then $C(X) \rtimes G$ has uniform property $\Gamma$.
\end{corollary}
\begin{proof}
Since crossed products of free minimal actions are simple by \cite[Theorem 2]{AS:free_minimal}, the conclusion follows by \Cref{thm:LTH_gamma} and \Cref{prop:LTH_crossed}.
\end{proof}

The second part of \Cref{thm_main:Vill_cp} can be recovered via \cite[Corollary E]{naryshkin:urpc}, which shows that all free minimal actions of FC groups---that is of groups with finite conjugacy classes---have the URP. Moreover, Naryshkin's theorem yields \emph{comparison} in the sense of  \cite[Definition 3.2]{kerr:af} for such actions. As a consequence, uniform property $\Gamma$ for the corresponding crossed products could also be deduced by \Cref{cor:sr1_tad_gamma} using the main result of \cite{LN:sr1}, where Li and Niu show that the crossed products arising from free minimal dynamical systems of amenable groups with the URP and comparison (or, more precisely, the more operator algebraic \emph{Cuntz comparison on open sets} \cite[Definition 2.21]{LN:sr1}) have stable rank one (see also \cite{BGK:sr1} for an extension of those techniques beyond the nuclear setting).

It is worth noting, however, that this `stable rank one approach' does not, a priori, fully recover \Cref{cor:cp}, as it requires comparison---we say \emph{a priori} since no example of topological dynamical system without comparison is currently known.

\bibliographystyle{amsplain}
\bibliography{tracial_rr0}

\begin{bibdiv}
\begin{biblist}

\bib{AS:free_minimal}{article}{
      author={Archbold, R.~J.},
      author={Spielberg, J.~S.},
       title={Topologically free actions and ideals in discrete
  {$C^\ast$}-dynamical systems},
        date={1994},
        ISSN={0013-0915,1464-3839},
     journal={Proc. Edinburgh Math. Soc. (2)},
      volume={37},
      number={1},
       pages={119\ndash 124},
         url={https://doi.org/10.1017/S0013091500018733},
}

\bib{BGK:sr1}{unpublished}{
      author={Bell, Jamie},
      author={Geffen, Shirly},
      author={Kerr, David},
       title={Stable rank one in nonnuclear crossed products},
        date={2025},
        note={preprint arXiv:2511.20132},
}

\bib{BP:rr0}{article}{
      author={Brown, Lawrence~G.},
      author={Pedersen, Gert~K.},
       title={{$C^\ast$}-algebras of real rank zero},
        date={1991},
        ISSN={0022-1236,1096-0783},
     journal={J. Funct. Anal.},
      volume={99},
      number={1},
       pages={131\ndash 149},
         url={https://doi.org/10.1016/0022-1236(91)90056-B},
}

\bib{CCEGSTW}{unpublished}{
      author={Carrión, Jos\'e~R.},
      author={Castillejos, Jorge},
      author={Evington, Samuel},
      author={Gabe, James},
      author={Schafhauser, Christopher},
      author={Tikuisis, Aaron},
      author={White, Stuart},
       title={Tracially complete {$C^\ast$}-algebras},
        note={preprint arXiv:2310.20594},
}

\bib{CETWW}{article}{
      author={Castillejos, Jorge},
      author={Evington, Samuel},
      author={Tikuisis, Aaron},
      author={White, Stuart},
      author={Winter, Wilhelm},
       title={Nuclear dimension of simple {$C^\ast$}-algebras},
        date={2021},
        ISSN={0020-9910,1432-1297},
     journal={Invent. Math.},
      volume={224},
      number={1},
       pages={245\ndash 290},
         url={https://doi.org/10.1007/s00222-020-01013-1},
}

\bib{CETW:gamma}{article}{
      author={Castillejos, Jorge},
      author={Evington, Samuel},
      author={Tikuisis, Aaron},
      author={White, Stuart},
       title={Uniform property {$\Gamma$}},
        date={2022},
        ISSN={1073-7928,1687-0247},
     journal={Int. Math. Res. Not. IMRN},
      number={13},
       pages={9864\ndash 9908},
         url={https://doi.org/10.1093/imrn/rnaa282},
}

\bib{EHT:sr1}{article}{
      author={Elliott, George~A.},
      author={Ho, Toan~M.},
      author={Toms, Andrew~S.},
       title={A class of simple {$C^*$}-algebras with stable rank one},
        date={2009},
        ISSN={0022-1236,1096-0783},
     journal={J. Funct. Anal.},
      volume={256},
      number={2},
       pages={307\ndash 322},
         url={https://doi.org/10.1016/j.jfa.2008.08.001},
}

\bib{ELN:Villadsen}{article}{
      author={Elliott, George~A.},
      author={Li, Chun~Guang},
      author={Niu, Zhuang},
       title={Remarks on {V}illadsen algebras},
        date={2024},
        ISSN={0022-1236,1096-0783},
     journal={J. Funct. Anal.},
      volume={287},
      number={7},
       pages={Paper No. 110547, 55},
         url={https://doi.org/10.1016/j.jfa.2024.110547},
}

\bib{EN:minimal}{article}{
      author={Elliott, George~A.},
      author={Niu, Zhuang},
       title={The {C{$^*$}}-algebra of a minimal homeomorphism of zero mean
  dimension},
        date={2017},
        ISSN={0012-7094,1547-7398},
     journal={Duke Math. J.},
      volume={166},
      number={18},
       pages={3569\ndash 3594},
         url={https://doi.org/10.1215/00127094-2017-0033},
}

\bib{EN:sbp}{unpublished}{
      author={Elliott, George~A.},
      author={Niu, Zhuang},
       title={On the small boundary property, $\mathcal{Z}$-stability, and
  bauer simplexes},
        date={2024},
        note={Preprint arXiv:2406.09748},
}

\bib{EN:propS}{unpublished}{
      author={Elliott, George~A.},
      author={Niu, Zhuang},
       title={On the small boundary property and $\mathcal{Z}$-absorption},
        date={2025},
        note={preprint arXiv:2504.03611},
}

\bib{engelking:dim_fininf}{book}{
      author={Engelking, Ryszard},
       title={Theory of dimensions finite and infinite},
      series={Sigma Series in Pure Mathematics},
   publisher={Heldermann Verlag, Lemgo},
        date={1995},
      volume={10},
        ISBN={3-88538-010-2},
}

\bib{ES:gamma}{article}{
      author={Evington, Samuel},
      author={Schafhauser, Christopher},
       title={Uniform property {$ \Gamma $} and finite dimensional tracial
  boundaries},
        date={2025},
     journal={Canad. J. Math., First View},
       pages={1\ndash 22},
}

\bib{FL:sr1}{article}{
      author={Fu, Xuanlong},
      author={Lin, Huaxin},
       title={Tracial oscillation zero and stable rank one},
        date={2025},
        ISSN={0008-414X,1496-4279},
     journal={Canad. J. Math.},
      volume={77},
      number={2},
       pages={563\ndash 630},
         url={https://doi.org/10.4153/S0008414X24000099},
}

\bib{Fu:sr1_rr0}{unpublished}{
      author={Fu, Xuanlong},
       title={From stable rank one to real rank zero: A note on tracial
  approximate oscillation zero},
        date={2025},
        note={preprint arXiv:2512.23911},
}

\bib{GGKNV:tr_amen}{article}{
      author={Gardella, Eusebio},
      author={Geffen, Shirly},
      author={Kranz, Julian},
      author={Naryshkin, Petr},
      author={Vaccaro, Andrea},
       title={Tracially amenable actions and purely infinite crossed products},
        date={2024},
        ISSN={0025-5831,1432-1807},
     journal={Math. Ann.},
      volume={390},
      number={3},
       pages={3665\ndash 3690},
         url={https://doi.org/10.1007/s00208-024-02833-9},
}

\bib{GiolKerr}{article}{
      author={Giol, Julien},
      author={Kerr, David},
       title={Subshifts and perforation},
        date={2010},
        ISSN={0075-4102,1435-5345},
     journal={J. Reine Angew. Math.},
      volume={639},
       pages={107\ndash 119},
         url={https://doi.org/10.1515/CRELLE.2010.012},
}

\bib{GP:diff_top}{book}{
      author={Guillemin, Victor},
      author={Pollack, Alan},
       title={Differential topology},
   publisher={AMS Chelsea Publishing, Providence, RI},
        date={2010},
        ISBN={978-0-8218-5193-7},
         url={https://doi.org/10.1090/chel/370},
        note={Reprint of the 1974 original},
}

\bib{haagerup:quasi}{article}{
      author={Haagerup, Uffe},
       title={Quasitraces on exact {$C^\ast$}-algebras are traces},
        date={2014},
        ISSN={0706-1994,2816-5810},
     journal={C. R. Math. Acad. Sci. Soc. R. Can.},
      volume={36},
      number={2-3},
       pages={67\ndash 92},
}

\bib{kerr:af}{article}{
      author={Kerr, David},
       title={Dimension, comparison, and almost finiteness},
        date={2020},
        ISSN={1435-9855},
     journal={J. Eur. Math. Soc.},
      volume={22},
      number={11},
       pages={3697\ndash 3745},
         url={https://doi.org/10.4171/jems/995},
}

\bib{KLTV}{article}{
      author={Kopsacheilis, Grigoris},
      author={Liao, Hung-Chang},
      author={Tikuisis, Aaron},
      author={Vaccaro, Andrea},
       title={Uniform property {$\Gamma$} and the small boundary property},
        date={2026},
        ISSN={0002-9947,1088-6850},
     journal={Trans. Amer. Math. Soc.},
      volume={379},
      number={1},
       pages={487\ndash 509},
         url={https://doi.org/10.1090/tran/9506},
}

\bib{KR:central}{article}{
      author={Kirchberg, Eberhard},
      author={R{\o}rdam, Mikael},
       title={Central sequence {$C^*$}-algebras and tensorial absorption of the
  {J}iang-{S}u algebra},
        date={2014},
        ISSN={0075-4102,1435-5345},
     journal={J. Reine Angew. Math.},
      volume={695},
       pages={175\ndash 214},
         url={https://doi.org/10.1515/crelle-2012-0118},
}

\bib{KS:sbp}{article}{
      author={Kerr, David},
      author={Szab\'{o}, G\'{a}bor},
       title={Almost finiteness and the small boundary property},
        date={2020},
        ISSN={0010-3616,1432-0916},
     journal={Comm. Math. Phys.},
      volume={374},
      number={1},
       pages={1\ndash 31},
         url={https://doi.org/10.1007/s00220-019-03519-z},
}

\bib{KW:ample}{unpublished}{
      author={Kopsacheilis, Grigoris},
      author={Winter, Wilhelm},
       title={Diagonal comparison of ample ${C}^\ast$-algebras},
        date={2025},
        note={preprint arXiv:2410.05967},
}

\bib{LN:sr1}{unpublished}{
      author={Li, Chun~Guang},
      author={Niu, Zhuang},
       title={Stable rank one of ${C}({X}) \rtimes {\Gamma}$},
        date={2020},
        note={preprint arXiv:2008.03361},
}

\bib{LW:mdim}{article}{
      author={Lindenstrauss, Elon},
      author={Weiss, Benjamin},
       title={Mean topological dimension},
        date={2000},
        ISSN={0021-2172},
     journal={Israel J. Math.},
      volume={115},
       pages={1\ndash 24},
         url={https://doi.org/10.1007/BF02810577},
}

\bib{mcduff:central}{article}{
      author={McDuff, Dusa},
       title={Central sequences and the hyperfinite factor},
        date={1970},
        ISSN={0024-6115,1460-244X},
     journal={Proc. London Math. Soc. (3)},
      volume={21},
       pages={443\ndash 461},
         url={https://doi.org/10.1112/plms/s3-21.3.443},
}

\bib{MvN:IV}{article}{
      author={Murray, F.~J.},
      author={von Neumann, J.},
       title={On rings of operators. {IV}},
        date={1943},
        ISSN={0003-486X},
     journal={Ann. of Math. (2)},
      volume={44},
       pages={716\ndash 808},
         url={https://doi.org/10.2307/1969107},
}

\bib{naryshkin:urpc}{unpublished}{
      author={Naryshkin, Petr},
       title={{U}{R}{P}, comparison, mean dimension, and sharp shift
  embeddability},
        date={2024},
        note={preprint arXiv:2410.01757},
}

\bib{niu:transformation}{article}{
      author={Niu, Zhuang},
       title={{$\mathcal{Z}$}-stability of transformation group {${
  C}^\ast$}-algebras},
        date={2021},
        ISSN={0002-9947,1088-6850},
     journal={Trans. Amer. Math. Soc.},
      volume={374},
      number={10},
       pages={7525\ndash 7551},
         url={https://doi.org/10.1090/tran/8477},
      review={\MR{4315611}},
}

\bib{niu:rc_mdim_urp}{article}{
      author={Niu, Zhuang},
       title={Comparison radius and mean topological dimension: {R}okhlin
  property, comparison of open sets, and subhomogeneous {$ C^\ast$}-algebras},
        date={2022},
        ISSN={0021-7670,1565-8538},
     journal={J. Anal. Math.},
      volume={146},
      number={2},
       pages={595\ndash 672},
         url={https://doi.org/10.1007/s11854-022-0205-8},
}

\bib{EPR}{article}{
      author={Ortega, Eduard},
      author={Perera, Francesc},
      author={R{\o}rdam, Mikael},
       title={The corona factorization property, stability, and the {C}untz
  semigroup of a {$C^\ast$}-algebra},
        date={2012},
        ISSN={1073-7928,1687-0247},
     journal={Int. Math. Res. Not. IMRN},
      number={1},
       pages={34\ndash 66},
         url={https://doi.org/10.1093/imrn/rnr013},
}

\bib{Pears:dim}{book}{
      author={Pears, A.~R.},
       title={Dimension theory of general spaces},
   publisher={Cambridge University Press, Cambridge, England-New
  York-Melbourne},
        date={1975},
}

\bib{RobRor:divisibility}{article}{
      author={Robert, Leonel},
      author={R{\o}rdam, Mikael},
       title={Divisibility properties for {$C^\ast$}-algebras},
        date={2013},
        ISSN={0024-6115,1460-244X},
     journal={Proc. Lond. Math. Soc. (3)},
      volume={106},
      number={6},
       pages={1330\ndash 1370},
         url={https://doi.org/10.1112/plms/pds082},
}

\bib{RW:revisited}{article}{
      author={R{\o}rdam, Mikael},
      author={Winter, Wilhelm},
       title={The {J}iang-{S}u algebra revisited},
        date={2010},
        ISSN={0075-4102,1435-5345},
     journal={J. Reine Angew. Math.},
      volume={642},
       pages={129\ndash 155},
         url={https://doi.org/10.1515/CRELLE.2010.039},
}

\bib{Sato:traces}{unpublished}{
      author={Sato, Yasuhiko},
       title={Trace spaces of simple nuclear {$C^\ast$}-algebras with
  finite-dimensional extreme boundary},
        date={2012},
        note={Preprint arXiv:1209.3000},
}

\bib{STW:99problems}{unpublished}{
      author={Schafhauser, Christopher},
      author={Tikuisis, Aaron},
      author={White, Stuart},
       title={Nuclear {$C^\ast$}-algebras: 99 problems},
        date={2025},
        note={preprint arXiv:2506.10902},
}

\bib{thiel:cuntz}{unpublished}{
      author={Thiel, Hannes},
       title={The {C}untz semigroup},
        date={2016},
        note={Lecture notes from a course at the University of Münster,
  \url{http://hannesthiel.org/wp-content/OtherWriting/CuScript.pdf}},
}

\bib{Thiel:ranks}{article}{
      author={Thiel, Hannes},
       title={Ranks of operators in simple {$C^\ast$}-algebras with stable rank
  one},
        date={2020},
        ISSN={0010-3616,1432-0916},
     journal={Comm. Math. Phys.},
      volume={377},
      number={1},
       pages={37\ndash 76},
         url={https://doi.org/10.1007/s00220-019-03491-8},
}

\bib{toms:growth}{article}{
      author={Toms, Andrew~S.},
       title={Flat dimension growth for {$C^\ast$}-algebras},
        date={2006},
        ISSN={0022-1236,1096-0783},
     journal={J. Funct. Anal.},
      volume={238},
      number={2},
       pages={678\ndash 708},
         url={https://doi.org/10.1016/j.jfa.2006.01.010},
}

\bib{TW:1}{article}{
      author={Toms, Andrew~S.},
      author={Winter, Wilhelm},
       title={The {E}lliott conjecture for {V}illadsen algebras of the first
  type},
        date={2009},
        ISSN={0022-1236,1096-0783},
     journal={J. Funct. Anal.},
      volume={256},
      number={5},
       pages={1311\ndash 1340},
         url={https://doi.org/10.1016/j.jfa.2008.12.015},
}

\bib{TWW:fin_dim}{article}{
      author={Toms, Andrew~S.},
      author={White, Stuart},
      author={Winter, Wilhelm},
       title={{$\mathcal{Z}$}-stability and finite-dimensional tracial
  boundaries},
        date={2015},
        ISSN={1073-7928},
     journal={Int. Math. Res. Not. IMRN},
      number={10},
       pages={2702\ndash 2727},
         url={https://doi.org/10.1093/imrn/rnu001},
}

\bib{Villadsen:1}{article}{
      author={Villadsen, Jesper},
       title={Simple {$C^\ast$}-algebras with perforation},
        date={1998},
        ISSN={0022-1236,1096-0783},
     journal={J. Funct. Anal.},
      volume={154},
      number={1},
       pages={110\ndash 116},
         url={https://doi.org/10.1006/jfan.1997.3168},
}

\bib{Winter:cd_II}{article}{
      author={Winter, Wilhelm},
       title={Covering dimension for nuclear {$C^\ast$}-algebras. {II}},
        date={2009},
        ISSN={0002-9947,1088-6850},
     journal={Trans. Amer. Math. Soc.},
      volume={361},
      number={8},
       pages={4143\ndash 4167},
         url={https://doi.org/10.1090/S0002-9947-09-04602-9},
}

\bib{Winter:pure}{article}{
      author={Winter, Wilhelm},
       title={Nuclear dimension and {$\scr{Z}$}-stability of pure {$\rm
  C^\ast$}-algebras},
        date={2012},
        ISSN={0020-9910,1432-1297},
     journal={Invent. Math.},
      volume={187},
      number={2},
       pages={259\ndash 342},
         url={https://doi.org/10.1007/s00222-011-0334-7},
}

\bib{WZ:nuc_dim}{article}{
      author={Winter, Wilhelm},
      author={Zacharias, Joachim},
       title={The nuclear dimension of {$C^\ast$}-algebras},
        date={2010},
        ISSN={0001-8708,1090-2082},
     journal={Adv. Math.},
      volume={224},
      number={2},
       pages={461\ndash 498},
         url={https://doi.org/10.1016/j.aim.2009.12.005},
}

\bib{Zhang:rr0_III}{article}{
      author={Zhang, Shuang},
       title={{$C^\ast$}-algebras with real rank zero and the internal
  structure of their corona and multiplier algebras. {III}},
        date={1990},
        ISSN={0008-414X,1496-4279},
     journal={Canad. J. Math.},
      volume={42},
      number={1},
       pages={159\ndash 190},
         url={https://doi.org/10.4153/CJM-1990-010-5},
}

\bib{Zhang:rr0}{article}{
      author={Zhang, Shuang},
       title={Matricial structure and homotopy type of simple
  {$C^\ast$}-algebras with real rank zero},
        date={1991},
        ISSN={0379-4024},
     journal={J. Operator Theory},
      volume={26},
      number={2},
       pages={283\ndash 312},
}

\end{biblist}
\end{bibdiv}

\end{document}